\documentclass[12pt,a4paper,oneside]{amsart} 
\usepackage{amsmath,amssymb,amsthm,graphicx,mathrsfs,url}
\usepackage[usenames,dvipsnames]{color}
\usepackage[colorlinks=true,linkcolor=Red,citecolor=Green]{hyperref}
\usepackage{amsxtra}
\usepackage{wasysym} 
\usepackage{graphicx}
\usepackage{subcaption}
\usepackage{array}
\newcolumntype{P}[1]{>{\centering\arraybackslash}m{#1}}
\usepackage{biblatex}
\usepackage{blindtext}

\def\?[#1]{\textbf{[#1]}\marginpar{\Large{\textbf{??}}}}

\def\smallsection#1{\smallskip\noindent\textbf{#1}.}
\let\epsilon=\varepsilon 

\newtheorem{theo}{Theorem}
\newtheorem*{theo*}{Theorem}
\newtheorem{prop}{Proposition}[section]
\newtheorem*{prop*}{Proposition}	
\newtheorem{defi}[prop]{Definition}
\newtheorem*{defi*}{Definition}

\newtheorem{assumption}{Assumption}
\newtheorem{lemm}[prop]{Lemma}
\newtheorem*{lemm*}{Lemma}

\numberwithin{equation}{section}

\newtheorem{rmk}{Remark}

\newtheorem{manualcorrreminner}{Corollary}
\newenvironment{manualcorr}[1]{%
  \renewcommand\themanualcorrreminner{#1}%
  \manualcorrreminner
}{\endmanualcorrreminner}

\DeclareMathOperator{\tr}{tr}

\usepackage{scalerel}

\newcommand\reallywidehat[1]{\arraycolsep=0pt\relax%
\begin{array}{c}
\stretchto{
  \scaleto{
    \scalerel*[\widthof{\ensuremath{#1}}]{\kern-.5pt\bigwedge\kern-.5pt}
    {\rule[-\textheight/2]{1ex}{\textheight}} 
  }{\textheight} %
}{0.5ex}\\           
#1\\                 
\rule{-1ex}{0ex}
\end{array}
}

\author{Tristan Humbert}
\email{tristan.humbert@ens.psl.eu}
\address{Ecole Normale Supérieure Ulm, Paris France 75005.}

\begin{document}
\begin{abstract}
We combine methods from microlocal analysis and dimension theory to study resonances with largest real part for an Anosov flow with smooth real valued potential. We show that the resonant states are closely related to special systems of measures supported on the stable manifolds introduced by Climenhaga in \cite{Cli}. As a result, we relate the presence of the resonances on the critical axis to mixing properties of the flow with respect to certain equilibrium measures and show that these equilibrium measures can be reconstructed from the spectral theory of the Anosov flow.
\end{abstract}
\title{First Ruelle resonance for an Anosov flow with smooth potential} 
\maketitle
\section{Introduction}
Let $(\mathcal M,g)$ be a smooth closed connected Riemannian manifold of dimension $n\geq 3$. We consider a smooth flow $\varphi_t$ on $\mathcal M$ and its generating vector field 
$$X(x):=\frac{d}{dt}\varphi_t(x)|_{t=0},\,\,\, x\in \mathcal M.$$
\begin{assumption}
\label{assumption}
We suppose that the flow is \emph{Anosov}, topologically transitive, that the stable and unstable bundles are orientable and of dimension $d_s$ and $d_u$ respectively. Furthermore, we consider a smooth, real valued potential $V$.
\end{assumption}
\subsection{Leading resonant state}
We study the operator $\mathbf P:= -X+V$ acting on specially designed \emph{anisotropic Sobolev spaces}. The set of eigenvalues of $\mathbf P$ on these spaces are called the \emph{Ruelle resonances}. Their set, which we will denote by $\mathrm{Res}$ is intrinsic to the Anosov flow and contains valuable dynamical meaning. Its understanding is essential to estimate the speed of decay of correlations, see for instance \cite{Liv,TsuZh}. More concretely, Ruelle resonances are defined by (see for instance \cite[Lemma 5.1]{DFG})
\begin{equation}
\label{eq:eqdiff}
\lambda \in \mathrm{Res} \iff \exists u\in \mathcal D'(\mathcal M)\setminus\{0\}, \ \mathrm{WF}(u)\subset E_u^*,\ \ (\mathbf P-\lambda)u=0. 
\end{equation}
Here, $\mathrm{WF}(u)$ denotes the wavefront set of the distribution $u$ and $E_u^*$ is the dual counterpart of the unstable bundle (see \eqref{eq:dual} for a precise definition). 

It is actually useful to study the operator $\mathbf P$ not only on functions but more generally on the space of $k$-forms in the kernel of the contraction by the flow: 
$$\begin{cases}\mathbf{P}: \mathscr E_0^k:=\{ u\in C^{\infty}(\mathcal M; \Lambda^k T^*\mathcal M) \mid \iota_{X}u=0\} \to \mathscr E_0^k
\\ \mathbf{P}\omega:=(-\mathcal L_X+V)\omega=-\frac{d}{dt}\varphi_t^*\omega|_{t=0}+ V\omega.
\end{cases}$$
We can also define resonances for $k$-forms by the equivalence \eqref{eq:eqdiff}\footnote{In this case, $u$ should be a $k$-current, more precisely in the dual of $\mathscr E_0^k$.}  and we denote their set by $\mathrm{Res}_k$. 
The decay rate of correlations is dictated by the resonances with large real part and hence of special interest is the study of resonances on the \emph{critical axis}:
\begin{equation}
\label{eq:criticalaxis}
\mathcal C_k:=\{\lambda \in \mathrm{Res}_k\mid \mathrm{Re}(\lambda)=\sup_{\mu \in \mathrm{Res}_k}\mathrm{Re}(\mu)\}.
\end{equation}
Starting from the relation
$$\forall \omega \in \mathscr E_0^k, \forall \lambda \in \mathbb C, \, \mathrm{Re}(\lambda)\gg 1, \ (\mathbf{P}_{|\mathscr E_0^k}-\lambda)^{-1}\omega=\int_0^{\infty}e^{t(\mathbf{P}-\lambda)}\omega dt, $$
we see that the position of $\mathcal C_k$ should be linked to the exponential growth of the norm of the propagator $e^{t\mathbf{P}}$ on relevant functional spaces (the so-called anisotropic Sobolev spaces that will be introduced in Subsection \ref{Anisotrop}). A form in $\mathscr E_0^k$ is a linear combination of $k$-wegdes of elements of $E_u^*$ and $E_s^*$ (see \eqref{eq:Decomposition} for the exact definition of these bundles). But an element in $E_s^*$ is contracted exponentially fast while an element of $E_u^*$ is expanded exponentially fast by the Anosov property. This means that in order to maximize the exponential growth of the norm, one should study the resolvent on $d_s$-forms.

Moreover, the exact location of $\mathcal C_0$ (resp. $\mathcal C_{d_s}$) is given by the $P(V+J^u)$ (resp. $P(V)$) where $P$ denotes the topological pressure and $J^u:=-\frac{d}{dt}\mathrm{det}(d\varphi_t(x)_{|E_u(x)})|_{t=0}$ is the unstable Jacobian. The corresponding eigenvectors (referred to as \emph{resonant states}) also bear dynamical significance. More precisely, the resonant states at the first resonance $P(V+J^u)$ (resp. $P(V)$) are linked to the system of \emph{leaf measures}. $m^s_{V+J^u}$ (resp. $m^s_V$). Leaf measures are systems of reference measures on stable and unstable leaves which are used to obtain the equilibrium state via a product construction. Their introduction goes back to Sinai \cite{Sin} for maps and Margulis \cite{Marg} for flows (when $V=0$). The measure of maximal entropy was obtained using leaf measures by Hamenstädt \cite{Ham} for geodesic flows in negative curvature and by Hasselblatt \cite{Has} for Anosov flows, see also \cite{Ham2} for an extension to non-zero potentials. Recently, Climenhaga, Pesin and Zelerowicz \cite{CPZ19,CPZ20,Cli} gave a new construction of leaf measures using dimension theory. Their construction extends to certain classes of partially-hyperbolic flows, see also the related works of Carrasco and Rodriguez-Hertz \cite{CarHer,CarHer2}.
\begin{theo}
\label{mainTheo}
Under Assumption \ref{assumption}, the critical axes for the action on $0$-forms and $d_s$-forms are given by:
$$
\mathcal C_0=\{\lambda \mid \mathrm{Re}(\lambda)=P(V+J^u)\},\ \mathcal C_{d_s}=\{\lambda \mid \mathrm{Re}(\lambda)=P(V)\}.$$
Moreover, 
$ P(V+J^u)$ $($ resp. $P(V))$ is a resonance called, the first resoance for the action on $0$-forms $($ resp. $d_s$-forms).

There is $\delta>0$ such that for any $k\neq d_s$, we have $\mathcal C_k\subset \{\lambda \mid \mathrm{Re}(\lambda)\leq P(V)-\delta\}$, i.e all other critical axis are to the left of $\mathcal C_{d_s}$, see Figure \ref{fig:fig}.

Moreover, all $\lambda \in \mathcal C_0$ $($resp. $\lambda \in \mathcal C_{d_s})$ have no Jordan block and the \emph{first resonance} $P(V+J^u)$ $($resp. $P(V))$ is simple:
\begin{equation}
\label{eq:firstres1}
\{u\in \mathcal D'(\mathcal M)\mid (\mathbf P -P(V+J^u))u=0, \ \mathrm{WF}(u)\subset E_u^*\}=\mathrm{Span}(\eta),
\end{equation}
where $\eta$ is a measure constructed in Theorem \ref{Construction} from the system of leaf measures $m^s_{V+J^u}$.
\begin{equation}
\label{eq:firstres2}
\{u\in \mathcal D'(\mathcal M; \Lambda^{d_s}(E_u^*\oplus E_s^*))\mid (\mathbf P -P(V))u=0, \ \mathrm{WF}(u)\subset E_u^*\}=\mathrm{Span}(m_V^s).
\end{equation}
\end{theo}
\begin{figure}
{\includegraphics[width=10cm]{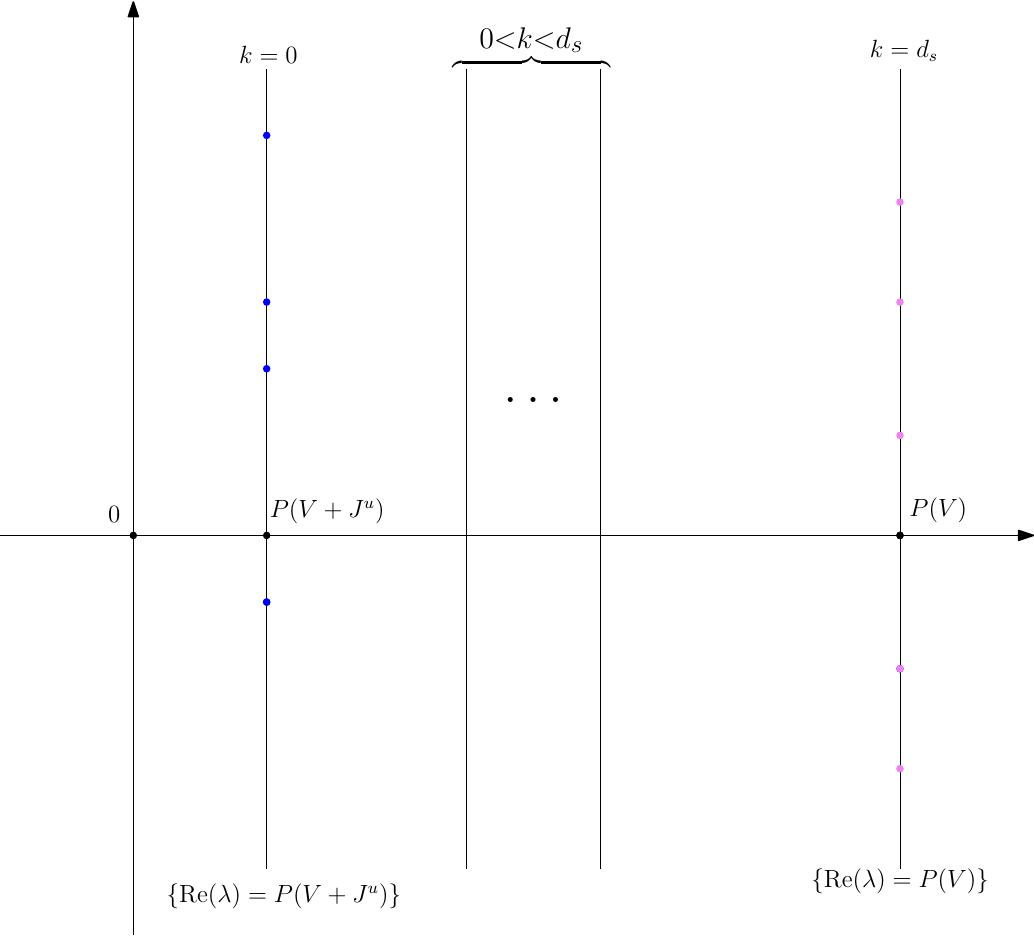}}
\caption{Critical axes for different values of $k$. According to Theorem \ref{mainTheo}, the resonances in purple cannot exist if the flow is weakly mixing with respect to $\mu_V$ and the resonances in blue cannot exist if the flow is weakly mixing with respect to $\mu_{V+J^u}$. The position of the critical axes for intermediate values of $k$ should be linked to the pressure on the span of largest Lyapunov exponents.}
\label{fig:fig}
\end{figure}

We note that for hyperbolic maps, similar results were already obtained by various authors, using the formalism of anisotropic Banach spaces. The study for the action of $0$-forms can be found it \cite[Theorem 7.5]{Ba} and for $d_s$-forms, it can be found in \cite[Theorem 5.1]{GouLi}. We also remark that Adam and Baladi used anisotropic techniques to study the first resonant state for an contact Anosov flows in dimension $3$ for the potential $V=-J^u$ in \cite{AdBa}\footnote{Because the unstable Jacobian is not smooth, this is actually, strictly speaking, out of the range of Theorem \ref{mainTheo}.}.
\subsection{Equilibrium states.}
The leaf measures constructed in \cite{CPZ19,CPZ20,Cli} can be used to reconstruct the \emph{equilibrium state} (i.e the unique invariant probability measure that maximises the variational principle recalled in  \eqref{eq:variational}). In our context, this means that the equilibrium state can be reconstructed from the spectral theory of the Anosov vector field $X$.  
Define the divergence of $X$ by the relation $\mathcal L_X \mathrm{vol}=\mathrm{div}_{\mathrm{vol}}(X) \mathrm{vol}$ for the Riemannian volume $\mathrm{vol}$. Then the $L^2$-adjoint of $\mathbf P$ acting on $0$-forms is $\mathbf P^*=X+V+\mathrm{div}_{\mathrm{vol}}(X).$ Applying Theorem \ref{mainTheo} to the adjoint,  gives two co-resonant states $\nu$ and $m^u_V$ and one has
\begin{equation}
\label{eq:system1}
\begin{cases}(-X+V-P(V+J^u))\eta=0,\quad  \mathrm{WF}(\eta)\subset E_u^*, 
\\ \ (X+V+\mathrm{div}_{\mathrm{vol}}(X)-P(V+J^u))\nu=0,\quad \mathrm{WF}(\nu) \subset E_s^*
\end{cases}
\end{equation}
as well as
\begin{equation}
\label{eq:system2}
\begin{cases}
(-\mathcal L_X+V-P(V))m_V^s=0, \quad \mathrm{WF}(m^s_V)\subset E_u^*,
\\ \  (\mathcal L_X+V-P(V))m_V^u=0,\quad \mathrm{WF}(m^u_V) \subset E_s^*.
\end{cases}
\end{equation}
The wavefront set bounds allow us to take the distributional pairing of the resonant and co-resonant states $\eta$ and $\nu$ (resp. $m^s_V$ and $m^u_V$). The resulting distribution is easily seen to be a measure invariant by the flow. The next theorem asserts that this measure is actually the equilibrium state  for the potential $V+J^u$ (resp. $V$).
As a consequence, the presence of other resonances on $\mathcal C_0$ (resp. $\mathcal C_{d_s})$ is linked to mixing properties of the flow. In the rest of the paper, we denote by $\mu_W$ the equilibrium measure associated to the Hölder continuous potential $W$.
\begin{theo}
\label{equilibriumtheo}
Under Assumption \ref{assumption}, one has 
\begin{equation}
\begin{cases}
\exists c>0, \ \mu_{V+J^u}=c\eta \times \nu,
\\ \exists c'>0, \ \mu_V=c'm^s_V\wedge \alpha \wedge m^u_V, \ \alpha(X)=1, \ \alpha(E_u\oplus E_s)=0.
\end{cases}
\end{equation}
Let $\chi$ be a cutoff $\chi\in C_c^{\infty}([0,T+\epsilon[, [0,1])$ such that $\chi \equiv 1$ on $[0,T]$ for any $\epsilon>0$ and for $T>0$ large enough. Then one has\footnote{The notation $f^{(k)}$ denotes the $k$-th convolution product of a function $f$ with itself.}:
\begin{itemize}
\item For any $f\in C^{\infty}(\mathcal M)$, 
 \begin{equation}
\label{eq:average}
-\lim_{k\to +\infty} \int_0^{+\infty}(\chi')^{(k)}(t )\big((\varphi_t)_*\nu\big)(f)dt=c\mu_{V+J^u}(f), \quad c>0.
\end{equation}
\item Fix an orientation of $E_u^*$ and let $\omega\in C^0(\mathcal M; \Lambda^{d_s}E_u^*)$ be a non-negative section which does not vanish identically. Then one has, for any $f\in C^{\infty}(\mathcal M)$, 
\begin{equation}
\label{eq:average2}
-\lim_{k\to +\infty} \int_0^{+\infty}(\chi')^{(k)}(t )\big((\varphi_t)_*(\omega \wedge \alpha \wedge m^u_V)\big)(f)dt=c\mu_{V}(f), \quad c>0.
\end{equation}
\end{itemize}
Finally, one has:
\begin{itemize}
\item The flow is topologically mixing if and only if it is not a constant time suspension if and only if 
$$\mathrm{Res}_0\cap \mathcal C_0=\{P(V+J^u)\},\quad \mathrm{Res}_{d_s}\cap \mathcal C_{d_s}=\{P(V)\}. $$
 \end{itemize}
\end{theo}
In particular for $V=0$, the theorem gives a construction of the SRB-measure for the action on $0$-forms (resp. the measure of maximal entropy for the action on $d_s$-forms) from a solution of \eqref{eq:system1} (resp. \eqref{eq:system2}) for $P(J^u)=0$ (resp. for $P(0)=\mathrm{h}_{\mathrm{top}}(\varphi_1)$).
\subsection{Ruelle zeta function.}
The resolvents acting on $\mathscr E_0^k$ are linked to the (weighted) \emph{Ruelle zeta function}:
\begin{equation}
\label{eq:Zeta}
\zeta_V(\lambda):=\prod_{{\gamma\in \Gamma^{\sharp}}}\big(1-e^{-(\lambda+V_{\gamma}) T_{\gamma}}\big), \ \ V_{\gamma}:=\frac{1}{T_{\gamma}}\int_0^{T_{\gamma}}V(\gamma(t))dt,
\end{equation}
where $\Gamma^{\sharp}$ is the set of primitive geodesics and $T_{\gamma}$ denotes the period of the closed geodesic $\gamma$. This function can be shown to be convergent and holomorphic in a half plane $\{\mathrm{Re}(\lambda)\gg 1\}$. In a celebrated paper \cite{GLP}, Giulietti, Liverani and Pollicott proved that the function $\zeta_R^V$ admits a meromorphic extension to the whole complex plane. Another proof, using microlocal analysis, was given by Dyatlov and Zworski in \cite{DyaZw}.

 More precisely, let us introduce the \emph{Poincaré} map $\mathcal P:\gamma\in\Gamma\mapsto \mathcal P_{\gamma}:=d\varphi_{-T_{\gamma}}|_{E_s\oplus E_u} $. The link between the resolvents on forms and the Ruelle zeta function is given by the \emph{Guillemin trace formula}. In particular, \cite[Equation (2.5)]{DyaZw} gives
%
%
\begin{equation}
\label{eq:zetta}
\zeta_V'(\lambda)/\zeta_V(\lambda)=\sum_{k=0}^{n-1}(-1)^{k+q} e^{-\lambda t_0}\tr^{\flat}\big(\varphi_{-t_0}^*\big((\mathbf{P})_{|\mathscr E_0^k}-\lambda)^{-1}\big) \big),\end{equation}
where the shift by a small time $t_0$ is a technicality to ensure that the pullbacked resolvent $\varphi_{-t_0}^*(\mathbf{P}-\lambda)^{-1}$ satisfies the wavefront set condition which makes its flat trace well defined (see \cite[Section 4]{DyaZw}). This shows that the meromorphic extension of $\zeta_V$ follows from the extension of the resolvent acting on the space $\mathscr E_0^k$ for any $k$ (plus some additional arguments). Moreover, we get the poles of $\zeta_V$ by studying poles of each resolvent. 

The study of the first pole of the Ruelle zeta function can be found in \cite[Theorem 9.2]{PaPo}.
\begin{theo*}[Parry-Pollicott]
Let $V$ be a Hölder continuous potential and suppose that the flow is Anosov and weakly topologically mixing. Then the Ruelle zeta function $\zeta_V$ is non zero and analytic in the half plane $\{Re(\lambda)\geq P(V)\}$ except for a simple pole at $\lambda=P(V)$, where $P(V)$ is the topological pressure of the potential $V$.
\end{theo*}

As a consequence of our first two theorems, we get an analogue of the theorem of Parry and Pollicott on the first pole of the Ruelle zeta function.
\begin{manualcorr}{1}[Critical axis of the zeta function]
\label{corro}
Suppose the flow to be topologically mixing, Anosov and let $V$ be a smooth real-valued potential, then the Ruelle zeta function $\zeta_V^R$ is analytic and non vanishing in the half plane $\{Re(\lambda)\geq P(V)\}$ except for a simple pole at $\lambda=P(V)$.
\end{manualcorr}
We note that a consequence of this corollary and a standard Tauberian argument is the following asymptotic growth:
$$\sum_{\gamma \in \Gamma, T_{\gamma}\leq t}e^{V_{\gamma}T_{\gamma}}=\sum_{\gamma \in \Gamma, T_{\gamma}\leq t}e^{\int_0^{T_\gamma}V(\gamma(t))dt}\sim\frac{e^{tP(V)}}{tP(V)}. $$
For $V=0$, this result is known as the Prime Orbit Theorem and can be found in \cite[Theorem 9.3]{PaPo}.
\subsection{Regularity of the pressure.}
Another consequence of our first theorems is the following regularity statement of the topological pressure. It was first established by Katok and al in \cite{KKPW} and Contreras in \cite{Con}.
\begin{manualcorr}{2}[Smoothness of the topological pressure]
\label{2}
Let $\mathcal V^{\infty}_{\mathrm{t}}$ denote the set of smooth transitive Anosov flows. Then it is an open set\footnote{This point follows from \cite[Proposition 1.6.30]{FishHas} which proves that topological transitivity is preserved by orbit conjugacy and the structural stability of Anosov flows, see \cite[Corollary 5.4.7]{FishHas}.} and the maps
\begin{align*}P_1: (X,V)\in \mathcal V^{\infty}_{\mathrm{t}}\times C^{\infty}(\mathcal M)\mapsto P_X(V),  
\\ P_2: (X,V)\in \mathcal V^{\infty}_{\mathrm{t}}\times C^{\infty}(\mathcal M)\mapsto P_X(V+J^u)
\end{align*}
where $P_X$ denotes the topological pressure for the flow induced by the vector field $X$, are smooth.

\end{manualcorr}
\subsection{Complex potential}
Consider a smooth complex valued potential $W=V+iU\in C^{\infty}(\mathcal M, \mathbb C)$ with $V,U\in C^{\infty}(\mathcal M, \mathbb C)$. While proving Theorem \ref{mainTheo}, we will obtain the following result.
\begin{manualcorr}{3}[Complex potential]
\label{complex}
For a smooth complex valued potential $W=V+iU$
\begin{equation}
    \label{eq:inclusion}
    \begin{split}
    &\mathrm{Res}_0(W)\subset \{\lambda \in \mathbb C\mid \mathrm{Re}(\lambda)\leq P(V+J^u)\},
    \\& \mathrm{Res}_{d_s}(W)\subset \{\lambda \in \mathbb C\mid \mathrm{Re}(\lambda)\leq P(V)\}.
    \end{split}
\end{equation}
Moreover, one has $\mathrm{Res}_0(W)\cap \{ \mathrm{Re}(\lambda)= P(V+J^u)\}\neq \emptyset$ and $\mathrm{Res}_{d_s}(W)\cap \{ \mathrm{Re}(\lambda)= P(V)\}\neq \emptyset$ if and only if $U$ satisfies
\begin{equation}
    \label{eq:arithm}
   \exists \theta \in \mathbb R, \ \forall \gamma\in \Gamma, \quad T_\gamma(U_{\gamma}-\theta)\in 2\pi \mathbb Z.
\end{equation}
\end{manualcorr}
The position of the first resonance for a general complex valued potential is, to the best of the author's knowledge, an open question. 

\smallsection{Outline of article}
{\color{blue}}
\begin{itemize}
\item In Section \ref{Preli} we recall some important features of Anosov flows and of the thermodynamical formalism. Then we will review microlocal methods for the study of Anosov flows. Finally, we will recall the construction of leaf measures.
\item In Section \ref{sec4}, we recall the definition of Ruelle resonances using a parametrix construction. Then, we define in Theorem \ref{Construction} a co-resonant state for the action on $0$-forms. This allows us to precisely locate the critical axis at $\{\mathrm{Re}(\lambda)=P(V+J^u)\}$.
The rest of the section is devoted to the study of resonant states on the critical axis and more precisely to the proofs of the results announced in the introduction.
\item In Section \ref{sec5} we prove the equivalent results for the action on $d_s$-forms. The strategies of the proofs remain the same but some additional care is needed when adapting certain arguments to this case.
\item In Section \ref{sec6} we give a proof of first part of Theorem \ref{mainTheo} and of Corollary \ref{2}.
\item In Section \ref{seccomplex}, we prove Corollary \ref{complex}.
\end{itemize}

\textbf{Acknowledgements.} The author would like to thank Colin Guillarmou and Thibault Lefeuvre for introducing him to to this problem and more generally to the field, as well as for their careful and precious guidance during the writing of this paper. 

The author would also like to thank Gabriel Paternain and Gabriel Rivière for making some comments on an earlier version of the paper and Julien Moy for helping with making Figure \ref{fig:fig}.

\section{Preliminaries}
\label{Preli}
\subsection{Anosov flow}
\label{anoosov}
Our main assumption on the flow is that it is \emph{Anosov}.
\begin{defi}[Anosov]
\label{Anosov}
The flow $\varphi_t$ is Anosov $($or uniformly hyperbolic$)$ if
\begin{itemize}
\item There is a continuous splitting of the tangent space
\begin{equation}
\label{eq:Decomposition}
T_x\mathcal M=E_u(x)\oplus E_s(x)\oplus \mathbb R X(x).  
\end{equation}
\item The decomposition is flow-invariant, meaning that
\begin{equation}
\label{eq:flowinv}
\forall t\in \mathbb R,\,\, E_u(\varphi_t(x))=(d\varphi_t)_x( E_u(x)),\,\,E_s(\varphi_t(x))=(d\varphi_t)_x( E_s(x)).
\end{equation}
\item There are uniform constants $C>0$ and $\theta>0$ such that for every $x\in M$, we have
$$|(d\varphi_t)_x(v_s)|_g\leq Ce^{-\theta t}|v_s|_g,\quad \quad \forall t\geq 0, \,\, \forall v_s\in E_s(x).$$
$$|(d\varphi_t)_x(v_u)|_g\leq Ce^{-\theta |t|}|v_u|_g,\quad \quad \forall t\leq 0, \,\, \forall v_u\in E_u(x).$$
\end{itemize}
We will denote by $d_u$ and $d_s$ the dimensions of $E_u$ and $E_s$ respectively.
\end{defi}

An important class of examples is given by geodesic flows on the unit tangent bundle $\mathcal M=SM$ of a negatively curved closed Riemannian manifold $M$.
We recall the stable manifold theorem, see \cite[Theorem 6.4.9]{KaHas}.

For all $x\in \mathcal M$, there exists immersed submanifolds 
\begin{equation}
\label{eq:W^{s,u}}
\mathcal W^{s,u}(x):=\{y\in \mathcal M\mid d(\varphi_t(x),\varphi_t(y))\to_{t\to \pm\infty}0\}, 
\end{equation}
where $+$ (resp. $-$) corresponds to $s$ (resp. $u$), called the (strong) stable (resp. unstable) manifolds, such that 
$ T_x \mathcal W^{s,u}=E_{s,u}.$
Moreover, $x\mapsto \mathcal W^{s,u}(x)$ are (Hölder continuous) foliations of $\mathcal M$. We also define the weak stable and unstable manifolds
\begin{equation}
\label{eq:W^{ws,wu}}
\mathcal W^{ws,wu}:=\{y\in \mathcal M\mid \exists t_0\in \mathbb R, \ d(\varphi_{t}(x),\varphi_{t+t_0}(y))\to_{t\to \pm\infty}0\}=\bigcup_{t\in \mathbb R}\varphi_t(\mathcal W^{s,u}(x)), 
\end{equation}
their tangent spaces are given respectively by $\mathbb R X\oplus E_s$ and $\mathbb R X\oplus E_u$.

A consequence of the existence of these (un)stable manifolds is the \emph{local product structure},  see \cite[Proposition 6.4.13]{KaHas}.

For any $x_0\in \mathcal M$, there exists a neighborhood $V$ of $x_0$ such that for any $\epsilon>0$, there is a $\delta>0$ such that
$$\forall x,y\in V, \ d(x,y)\leq \delta \ \Rightarrow \ \exists |t|\leq \epsilon, \ \ \mathcal W_{\epsilon}^{u}(\varphi_t(x))\cap \mathcal W_{\epsilon}^{s}(y)=:\{[ x,y]\}, $$
where we denoted by $\mathcal N_{\epsilon}$ the $\epsilon$-ball of the manifold $\mathcal N$. The point $[x,y]$ is called the Bowen bracket of $x$ and $y$ and  $t(x,y)$ is the Bowen time. 
For $q\in \mathcal M$, we define a local rectangle to be
\begin{equation}
\label{eq:R}
R_q:=\{\Psi_q(x,y):=[x,y]\mid x\in \mathcal W^u(q)\cap B(q,\delta), \ y\in \mathcal W^{ws}(q)\cap B(q,\delta)\}.
\end{equation}

\subsection{Thermodynamical formalism}
\label{Thermo}
We recall here the main features from the thermodynamical formalism we will need. The two main objects are the \emph{topological pressure} and the \emph{equilibrium state}. Consider a Hölder continuous  and real valued  potential $V$.

We first recall the \emph{variational principle} (see \cite[Theorem 9.3.4]{FishHas}), which we will state in the case of smooth flows.

Let $(\mathcal M,g)$ be a closed Riemannian manifold and $\varphi_t$ be a smooth flow on $\mathcal M$, let $V:\mathcal M \to \mathbb R$ be a Hölder continuous potential, then
\begin{equation}
\label{eq:variational}
P(\varphi_1,V):=\sup_{\mu \in \mathfrak M(\varphi)}\left(h_{\mu}(\varphi_1)+\int_{\mathcal M}Vd\mu\right), 
\end{equation}
where $h_{\mu}$ is the metric entropy and $\mathfrak M(\varphi)$ is the set of invariant-probability Borel measures and $P(\varphi_1,V)$ is the \emph{toplogical pressure} associted to $V$.

Now, we can define \emph{equilibrium state} as a measure that achieves the maximum, the existence and uniqueness of such a measure can be obtained under Assumption \ref{assumption}, we will use the following result  (see \cite[Theorem 7.3.6]{FishHas} and \cite[Theorem 3.3]{BoRu}).

The equilibrium state associated to $V$ is unique, ergodic and has full support. If the flow is topologically mixing, then the flow is weak mixing with respect to the equilibrium state $\mu_V$. 

In the case of an Anosov flow, an equivalent characterization of equilibrium state is given by the \emph{Gibbs property} (see \cite[Theorem 4.3.26]{FishHas}). Indeed, $\mu$ is the equilibrium state for $V$ if and only if
\begin{equation}
\label{eq:Gibbs}\forall \delta>0, \ \exists C>0, \forall t>0, \forall q\in \mathcal M,\quad C^{-1}\leq \mu(B_t(q,\delta))e^{tP(V)-S_tV(q)}\leq C. 
\end{equation}
Here, $B_t(q,\delta)$ denotes the \emph{Bowen ball} defined in \eqref{eq:BowenBall} and $S_tV(q):=\int_0^tV(\varphi_s q)ds.$

We define a special potential called the \emph{unstable Jacobian} by the following formula:
\begin{equation}
\label{eq:Ju}
J^u(x):=-\frac{d}{dt}\mathrm{det}(d\varphi_t(x)_{|E_u(x)})|_{t=0}=:-\frac{d}{dt}j_t(x)|_{t=0}=-\frac{d}{dt}\ln j_t(x)|_{t=0},
\end{equation}
where the determinant is taken with respect to the Riemannian measure $\mathrm{vol}$.
 
The \emph{equilibrium state} associated to the unstable Jacobian is the Sinai-Ruelle-Bowen (SRB) measure whose pression vanishes : $P(J^u)=0$, see \cite[Corollary 7.4.5]{FishHas}.

\subsection{Anisotropic spaces}
\label{Anisotrop}

To a given decomposition \eqref{eq:Decomposition}, we can associate a corresponding splitting of the cotangent space. This will be more natural as we will use microlocal analysis.

There is a continuous slitting 
$
T^*_xM=E_u^*(x)\oplus E_s^*(x)\oplus E_0^*(x), 
$
defined by 
\begin{equation} 
\label{eq:dual}
E_s^*(x)\big(E_s(x)\oplus \mathbb R {X}(x)\big)=0,\,\,E_u^*(x)\big(E_u(x)\oplus \mathbb R {X}(x)\big)=0,\,\,E_0^*(x)\big(E_s(x)\oplus E_u(x)\big)=0. 
\end{equation}
Moreover, this decomposition is flow-invariant and there exists constants $C,\theta>0$ such that, uniformly in $x\in M$, we have
$$|(d\varphi_{-t})^T_x(\xi_s)|\leq Ce^{-\theta t}|\xi_s|,\quad \quad \forall t\geq 0, \,\, \forall \xi_s\in E_s^*(x).$$
$$|(d\varphi_{-t})^T_x(\xi_u)|\leq Ce^{-\theta |t|}|\xi_u|,\quad \quad \forall t\leq 0, \,\, \forall \xi_u\in E_u^*(x).$$

The following result, which is due to Faure, Roy, Sjöstrand in \cite{Fau08} constructs an anisotropic order function. 

 There exists an order function $m(x,\xi)$ taking its values in $[-1,1]$ and an escape function 
$G_m(x,\xi):=m(x,\xi)\log |\xi|, $ such that:
\begin{itemize}
\item The order function $m(x,\xi)$ only depends on the direction $\xi/|\xi|\in S^*M$ for $|\xi|\geq 1$ and is equal to $1$ $($resp. $-1)$ in a conical neighborhood of $E_s^*$ $($resp. $E_u^*)$.
\item The escape function decreases along trajectories, that is
$$\exists R>0,|\xi|\geq R,\,\, \mathbf{X}(G_m)(x,\xi)\leq 0\footnote{For this inequality, one should work with an adapted metric and $\mathbf{X}$ denotes the symplectic lift of the vector field $X$.}. $$
\end{itemize}

We fix now an order function $m$ and consider the corresponding symbol class.
We refer to \cite[Appendix]{Fau08} for the detailed construction. What is important to understand here is that the order function $m(x,\xi)$ constructed below gives rise to a symbol class $S^{m(x,\xi)}$ on which we can perform quantization. These quantized symbols are called pseudo-differential operators of order $m(x,\xi)$ and can be viewed as bounded operator from the anisotropic Sobolev space $H^{m(x,\xi)}$ to $L^2(M)$. The anisotropic Sobolev space $H^{m(x,\xi)}$ is defined by means of an \emph{elliptic} operator in the anisotropic symbol class $S^{m(x,\xi)}_{\rho}$ (see \cite[Appendix, definition 7]{Fau08}). 

The symbol $\exp(G_m)$ belongs to the anisotropic class $S^{m(x,\xi)}_{\rho}$, for every $\rho<1$ and if we fix a quantization $\mathrm{Op}$ (see for instance \cite[Chapter 4]{Zw})
$$\hat A_m:=\mathrm{Op}(\exp(G_m)) $$
is a pseudo-differential operator which is elliptic and, up to changing the symbol by a $O(S^{m(x,\xi)-(2\rho-1)})$ term, it can be made formally self-adjoint and invertible on $C^{\infty}(M)$.
For $s\in \mathbb R$, we define the Sobolev space of order $sm(x,\xi)$ to be 
$
\mathcal H^s:=\hat A_{sm}^{-1}(L^2(M)).$
In the following, the $L^2$ spaces will be associated to the Riemannian volume form $\mathrm{vol}$.
The following properties hold:
\begin{itemize}
\item The space $\mathcal H^s$ is a Hilbert space with inner product
$$(\varphi_1,\varphi_2)_{\mathcal H^s}:=(\hat A_{sm} \varphi_1,\hat A_{sm} \varphi_2)_{L^2}, $$
makes $\hat A_{sm}$ a unitary operator from $\mathcal H^s$ to $L^2$.
\item For a pseudo-differential operator $A\in \Psi^{sm(x,\xi)}$, $A$ is an unbounded operator on $L^2(M)$ with domain given by $\mathcal H^s$.
\item The space $\mathcal H^{-s}$ can be identified to the dual of $\mathcal H^s$ by
\begin{equation}
\label{eq:dual2}
(\varphi,\psi)_{\mathcal H^s\times \mathcal H^{-s}}:=(\hat A_{sm} \varphi, \hat A_{sm}^{-1} \psi)_{L^2}
\end{equation}
and the duality extends the usual $L^2$-pairing.
\item If $f\in C^{\infty}(\mathcal M)$, then for any $ \varphi \in \mathcal H^s, \ \psi\in \mathcal H^{-s}$,
\begin{equation}
\label{eq:compute}
(f\varphi,\psi)_{\mathcal H^s\times \mathcal H^{-s}}=(\varphi,\bar f\psi)_{\mathcal H^s\times \mathcal H^{-s}}
\end{equation}
\end{itemize}

\label{Ruelle}
We now use the microlocal techniques introduced in \cite{Fau08,Fau10} to define the \emph{Ruelle resonances}. We first define the \emph{transfer operator} associated to the Anosov flow. 

The transfer operator $e^{t \bf P}: L^2(\mathcal M, \mathrm{vol}) \to L^2(\mathcal M,\mathrm{vol})$ (where $\mathbf{P}:=-X+V$) is given, for any $f\in C^{\infty}(M)$, by
$$e^{t \bf P}f(x):=\exp\left( \int_0^t V(\varphi_{-s}(x))ds\right)f(\varphi_{-t}(x))=:\exp\big(S_tV(\varphi_{-t}x)\big)f(\varphi_{-t}(x)). $$
We define the exponential growth in $L^2$-norm of $e^{t\mathbf{P}}$ by
$$C_0(\mathbf{P})=C_0:=\limsup_{t\to +\infty}\frac 1t\ln \|e^{t\mathbf{P}}\|_{L^2\to L^2}.  $$
The fact that $C_0$ is finite  is a consequence of $e^{t\mathbf{P}}$ being a semi-group.

We obtain the existence of the meromorphic extension of the resolvent to the whole complex plane and the fact that Ruelle resonances are contained in $\{\mathrm{Re}(\lambda)\leq C_0\}$. We refer to \cite[Theorem 9.11]{Lef} for a proof of the following theorem.
\begin{theo}[Faure Sjöstrand]
\label{Resolvent+V}
There exists $c>0$ such that for any $s> 0$, we have, for any $\lambda$ such that $\mathrm{Re}(\lambda)> -cs+C_0$, that $\mathbf{P}-\lambda$ is Fredholm of index $0$ as an operator
\begin{equation}
\label{eq:X}
\mathbf{P}-\lambda: \mathrm{Dom}(\mathbf{P})\cap \mathcal H^s=\{u\in \mathcal H^s \mid \mathbf{P} u\in \mathcal H^s\}\to \mathcal H^s.
\end{equation}
Moreover, the resolvent 
\begin{equation}
\label{eq:resolv}
R_+(\lambda)=(\mathbf{P}-\lambda)^{-1}=(-X+V-\lambda)^{-1}: \mathcal H^s\to \mathcal H^s 
\end{equation}
is well defined, bounded and holomorphic for $\{\mathrm{Re}(\lambda)> C_0\}$ and has a meromorphic extension to $\{\mathrm{Re}(\lambda)> -cs+C_0\}$ which is independent of any choice. Thus, the resolvent, viewed as an operator $C^{\infty}(M)\to \mathcal D'(M)$, has a meromorphic extension to the whole complex plane. The poles of this extension are called the \emph{Ruelle resonances}, they are located in $\{\mathrm{Re}(\lambda)\leq C_0\}$. On $\{\mathrm{Re}(\lambda)> C_0\}$, one has 
\begin{equation}
\label{eq:resolventeintegrale}
(\mathbf{P}-\lambda)^{-1}f=\int_{0}^{+\infty}e^{t(\mathbf{P}-\lambda)}fdt=\int_{0}^{+\infty}\!\!\exp\left( \int_0^t V(\varphi_{-s}(x))ds\right)e^{-t({X}+\lambda)}fdt.
\end{equation}
\end{theo}

\begin{rmk}
\label{remm}
The previous construction can be made for a general smooth\footnote{We insist on the fact that smoothness is important and this leads to some technical difficulties when working on forms.} Hermitian bundle $\mathscr E$ and in particular for the bundle of forms in the kernel of the contraction, see \cite[Appendix C]{DyaZw} for more details. We will not specify the dependance in $\mathscr E$ in the rest of the paper.
\end{rmk}

We sum up in the end of the subsection the characterization of generalized eigenfunctions: the \emph{resonant states.} 

A complex number $\lambda_0$ is a Ruelle resonance if and only if there exists a distribution $u\in \mathcal D'(\mathcal M)$ with wavefront set contained in $E_u^*$ (see for instance \cite[Lemma 5.1]{DFG}) such that $(\mathbf P -\lambda_0)u=0$, we will then say that $u$ is a \emph{resonant state}. We will sometime write $\mathrm{Res}$ the set of Ruelle resonances and $\mathrm{Res}_{\lambda_0}$ the set of resonant states associated to $\lambda_0$.
\begin{equation}
\label{eq:resonant}
\mathrm{Res}_{\lambda_0}:=\{u\in \mathcal D'(\mathcal M)\mid (\mathbf P -\lambda_0)u=0, \ \mathrm{WF}(u)\subset E_u^*\}. \end{equation}

We have a corresponding version for the co-resonant states (defined in the next subsection). The wavefront set condition then becomes $\mathrm{WF}(u)\subset E_s^*$ and this will be the version of the proposition we will use in the proof of Theorem \ref{Construction}.

We finish the subsection by discussing \emph{generalized resonant states} and the presence of Jordan blocks.
More precisely, if we consider the meromorphic extension $R_+(\lambda)$  constructed in Theorem \ref{Resolvent+V}, then $\lambda_0\in \mathrm{Res}$ is and only if $\lambda_0$ is a pole of the meromorphic extension. In this case, the spectral projector at $\lambda_0$ is 
$$ \Pi^{+}_{\lambda_0}=\frac{1}{2i\pi}\int_{\gamma}R_+(z)dz,$$
where $\gamma$ is a small loop around $\lambda_0$. Moreover, we can use the analytic Fredholm theorem to deduce that the resolvent has the following expansion:
$$R_+(\lambda)=R_+^H(\lambda)+\sum_{j=1}^{N(\lambda_0)}\frac{(\mathbf{P} -\lambda_0)^{j-1}\Pi_{\lambda_0}^+}{(\lambda-\lambda_0)^j}, $$
where $R_+^H(\lambda)$ is the holomorphic part near $\lambda_0$. The \emph{generalized resonant states} are:
\begin{equation}
\label{eq:Jordan}
\mathrm{Res}_{\lambda_0,\infty}:=\Pi_{\lambda_0}^+(\mathcal H^s)=\Pi_{\lambda_0}^+(C^{\infty}(\mathcal M))=\{u\in \mathcal H^s\mid (\mathbf{P}-\lambda_0)^{N(\lambda_0)}u=0\}.
\end{equation}
\begin{rmk}
\label{JordanBlock}
We will say that the Ruelle resonance $\lambda_0$ has no Jordan block is $\mathrm{Res}_{\lambda_0,\infty}=\mathrm{Res}_{\lambda_0}$. Note that if $N(\lambda_0)=1$, i.e the resolvent has a pole of order at most $1$, then there is no Jordan block. This will be used in Lemma \ref{pole} and Lemma \ref{realnorm}  to show that resonances on the critical axes have no Jordan blocks.
\end{rmk}

\subsection{Equilibrium states from dimension theory}
\label{sec:Climenhaga}
In this section, we recall Climenhaga's construction from \cite{Cli} of leaf measures $m^u_V$ and $m^s_V$ using dimension theory.

There are two main ways equilibrium states are defined. The first one is through the use of \emph{Markov partitions} and the second one is via the use of the \emph{specification property}. A third approach is given by dimension theory: the goal is to generalize the idea of \emph{Haussdorff dimension} and \emph{Haussdorff measure} to a more dynamical setting. We recall the definition of the Haussdorff dimension for a metric space $(X,\delta)$. For $d\geq 0$ and $\epsilon>0$, define the $d$-dimensional Haussdorff measure by
$$H^d_{\epsilon}(S):=\inf \{\sum_{i=0}^{\infty} \mathrm{diam}(U_i)^d \mid S\subset \bigcup_{i\in \mathbb N}U_i, \ \mathrm{diam}(U_i)<\epsilon\},$$
for any subset $S$ and where the infimum is taken over all countable covering of sets $U_i$ with diameter less than $\epsilon$. We define an outer measure by taking $$H^d(S)=\lim_{\epsilon\to +\infty} H^d_{\epsilon}(S)\in [0,+\infty].$$
We then define the Haussdorff dimension of $X$ to be
$$\mathrm{dim}_{\mathrm{Hauss}}(X):=\inf\{d\geq 0 \mid H^d(X)=+\infty\}=\sup\{d\geq 0 \mid H^d(X)=0\}. $$

The idea of Climenhaga, Pesin and Zelerowicz (already present in essence in \cite{Ham} and \cite{Has}, where the case of the measure of maximal entropy was treated) was to replace the sets with small diameters by more dynamical objects, namely, coverings should be made of \emph{Bowen balls} (defined below) and we should let time $t\to +\infty$.
\begin{equation}
\label{eq:BowenBall}
B_t(x,r):=\{y\in \mathcal M \mid \max_{s\in[0,t]}d(\varphi_s x,\varphi_s y)<r\}. 
\end{equation}


Let $\delta_0>0$ be the size of the local (un)stable manifolds, then fix thereafter $\delta\in ]0,\delta_0[$  and define $\mathcal W^{\bullet}(x,\delta):=B(x,\delta)\cap \mathcal W^{\bullet}(x)$ where $\bullet=u,s,ws,wu$. We will not always specify in the rest of the paper the dependance in $\delta$ if it is not relevant to the argument.

Let $x\in \mathcal M$, consider $Z\subset  \mathcal W^u(x,\delta)$. Define, for $T>0$,
$$ \mathbb E(Z,T):=\{\mathscr E \subset M\times [T,+\infty[, \ Z\subset \bigcup_{(x,t)\in \mathscr E}B_t(x,r, \mathcal W^u(x,\delta))\}. $$
Let $\alpha \in \mathbb R$, then we define a measure $m^{\alpha}_{ \mathcal W^u(x,\delta)}=m^{\alpha}$ by putting
$$m^{\alpha}(Z)=\lim_{T\to+\infty}\inf_{\mathscr E\in \mathbb E(Z,T)}\sum_{ (x,t)\in \mathscr E}e^{S_tV(x)-t\alpha}. $$

We then retrieve the Caratheodory dimension as a treshold just like in the case of the Haussdorff measure, this is a result due to Pesin, see \cite[Prop. 1.1 and 1.2]{Pes}. We have moreover, in this case that the measure for $\alpha=\alpha_C$ defines a Borel measure, see \cite[Lemm. 2.14]{Cli}.

The measure $m^{\alpha}$ defined above is an outer measure for any $\alpha \in \mathbb R$ and 
$$P(V):=\inf \{\alpha \mid m^{\alpha}( W(x,\delta ))=+\infty\}=\sup\{\alpha \mid m^{\alpha}( \mathcal W^u(x,\delta))=0\}. $$
Moreover, $m^{P(V)}$ is a Borel measure, denoted by $m^u_{x}$ and called the leaf measure. It satisfies
\begin{equation}
\label{eq:measure}
\forall Z \subset \mathcal W^u(x_1,\delta), \
m^u_x(Z)=\lim_{T\to+\infty}\inf_{\mathscr E\in \mathbb E(Z,T)}\sum_{ (x,t)\in \mathscr E}e^{S_tV(x)-tP(V)}.
\end{equation}

Up until now, we have defined a system of (unstable) leaf measures $\{m^u_x\mid  x\in \mathcal M\}$ satisfying the two following conditions:
\begin{itemize}
\item Support: each measure $m^u_x$ is supported in $\mathcal W^u(x,\delta)$.
\item Compatibility: if $Z\subset \mathcal W^u(x_1,\delta)\cap \mathcal W^u(x_2,\delta)$ is a Borel set, then the two measures agree, i.e $m^u_{x_1}(Z)=m^u_{x_2}(Z)$.
\end{itemize}
The set of measures defined in \eqref{eq:measure} has actually two more important properties: it is $\varphi_t$-conformal and behaves naturally with holonomies (see \cite[Theo 3.4]{Cli}).

The system of (unstable) leaf measures $\{m^u_x\mid  x\in \mathcal M\}$ defined in \eqref{eq:measure} is $\varphi_t$-conformal, namely, for any $x\in \mathcal M$ and $t\in \mathbb R$, the measures $(\varphi_t)_*m^u_x$ and $m^u_{\varphi_t x}$ are equivalent and more precisely:
\begin{equation}
\label{eq:change of variable}
m^u_{\varphi_t x}(\varphi_t Z)=\int _Z e^{tP(V)-S_tV(z)}dm^u_x(z). 
\end{equation}
In terms of Radon-Nikodym derivatives, we have
\begin{equation}
\label{eq:RN}
\frac{d((\varphi_{-t})_*m^u_{\varphi_t x})}{dm^u_x}(z)=e^{tP(V)-S_tV(z)},\quad \frac{d((\varphi_{t})_*m^u_{ x})}{dm^u_{\varphi _tx}}(\varphi_t z)=e^{S_tV(z)-tP(V)}.
\end{equation} 

We define the notion of \emph{holonomy} between (weak-un)stable leaves.

Given $\mathcal W^u(x_1,\delta),\mathcal W^u(x_2,\delta)$ for $x_1,x_2\in \mathcal M$, a \emph{weak-stable $\delta$-holonomy} between $\mathcal W^u(x_1,\delta)$ and $\mathcal W^u(x_2,\delta)$ is a homeomorphism $\pi:\mathcal W^u(x_1,\delta)\to \mathcal W^u(x_2,\delta)$ such that $\pi(z)\in \mathcal W^{ws}(z,\delta)$ for all $z\in \mathcal W^u(x_1,\delta)$. 

To conclude, we give the change of variable formula for holonomies. We first introduce a useful function.
Let $\delta_0>0$ be small enough, given $x\in \mathcal M$ and $y\in \mathcal W^{ws}(x,\delta_0)$, define
\begin{equation}
\label{eq:w+}
w_V^+(x,y)=S_{t(x,y)}V(x)+t(x,y)P(V)+\int_0^{+\infty}(V(\varphi_{t(x,y)+\tau}(x))-V(\varphi_{\tau}(y)))d\tau. 
\end{equation}
Here, $t(x,y)$ is the Bowen time and the integral converges because $d(\varphi_{t+\tau}x,\varphi_{\tau}y)\to 0$ exponentially fast and $V$ is Hölder continuous.

We define for $x\in \mathcal M$ and $y\in \mathcal W^{wu}(x,\delta_0)$ the quantity 
$$w_V^-(x,y)=-S_{t(x,y)}V(x)+t(x,y)P(V)+\int_0^{+\infty}(V(\varphi_{t-\tau}x)-V(\varphi_{-\tau}y))d\tau. $$

We note that we have the special values:
$$\forall y\in \mathcal W^s(x), \quad w^+_V(x,y)=\int_0^{+\infty}(V(\varphi_{\tau}x)-V(\varphi_{\tau}y))d\tau, $$
as well as
\begin{equation}
\label{eq:floww}
w^+(x,\varphi_t x)= S_{t}V(x)-tP(V).
\end{equation}
Finally, we have the cocycle relation 
\begin{equation}
\label{eq:cocycle}
\forall y,z\in \mathcal W^{ws}(x,\delta_0/2), \quad w^+_V(x,y)=w^+_V(x,z)+w^+_V(z,y).
\end{equation}

The holonomy we will mostly be interested in will be given by
$\pi: \mathcal W^u(q,\delta)\to  \mathcal W^u(p,\delta), \quad \pi(x):=[x,p]$ and we can now state the second change of variable theorem, see \cite[Theorem 3.4]{Cli} for more details.

Consider the system of (unstable) leaf measures $\{m^u_x\mid  x\in \mathcal M\}$ defined in \eqref{eq:measure}. Let $\mathcal W^u(x_1,\delta),\mathcal W^u(x_2,\delta)$ for $x_1,x_2\in \mathcal M$ and let $\pi:\mathcal W^u(x_1,\delta)\to \mathcal W^u(x_2,\delta)$ be a weak-stable $\delta_0$ holonomy. Then the measures $\pi_*m_{x_1}^u$ and $m_{x_2}^u$ are equivalent, and we have
\begin{equation}
\label{eq:COVpi}
\frac{d(\pi_*m^u_{x_1})}{dm^u_{x_2}}(\pi(z))=e^{w^+_V(z,\pi z)}.
\end{equation}

\section{Resolvent acting on functions}
\label{sec4}
In this section we study the action of $\mathbf{P}$ on functions. The goal is to locate the critical axis, show that it is given by $\{\mathrm{Re}(\lambda)=P(V+J^u)\}$ and to study the co-resonant states associated to resonances on this axis.
\subsection{Construction of the co-resonant state}
\label{duality}
We first prove that $P(V+J^u)$ is indeed a resonance, called the \emph{first resonance}. In the case of a null potential ($V=0$), this is trivial as constant functions lie trivially in the kernel of the flow. In this case, the equilibrium state (which is usually referred to as the SRB measure) can be identified as the co-resonant state for the Ruelle resonance $0$.

Following Section \ref{Ruelle}, we can define Ruelle resonances for the $L^2$-adjoint $\mathbf{P}^*=X+V-\mathrm{div}_{\mathrm{vol}}(X)$. The resonant states for $\mathbf{P}^*$ are referred to as \emph{co-resonant} states for $\mathbf{P}$ and their span has the same dimension as the the span of the resonant states. Equivalently, we can define the adjoint by duality, using relation \eqref{eq:dual2}:
\begin{equation}
\label{eq:duality}
\forall f,g\in C^{\infty}(\mathcal M), \quad (\mathbf P f,g)_{\mathcal H^s\times \mathcal H^{-s}}=( f,\mathbf P^*g)_{\mathcal H^s\times \mathcal H^{-s}},
\end{equation}
from which we can define an unbounded operator 
\begin{equation}
\label{eq:X*}
\mathbf{P}^*: \mathrm{Dom}(\mathbf{P}^*)\cap \mathcal H^{-s}=\{u\in \mathcal H^{-s} \mid \mathbf{P}^* u\in \mathcal H^{-s}\}\to \mathcal H^{-s}.
\end{equation}
Now, a direct adaptation of \cite[lemma 5.3]{GBGHW} yields.
\begin{lemm}[Co-resonant states]
\label{spectral projector}
Let $\lambda\in \mathbb C$, then $\lambda$ is a Ruelle resonance for $\mathbf{P}$ if and only if $\overline{\lambda}$ is a Ruelle resonance for $\mathbf{P}^*$. In this case, the space of resonant and co-resonant states have the same dimension $m$. If we consider $u_1,\ldots, u_m\in \mathcal D'_{E_u^*}(\mathcal M):=\{u\in \mathcal D'\mid \mathrm{WF}(u)\subset E_u^*\}$ and $v_1,\ldots, v_m\in \mathcal D'_{E_s^*}(\mathcal M):=\{u\in \mathcal D'\mid \mathrm{WF}(u)\subset E_s^*\}$ two basis of the resonant states for $\mathbf{P}$ and $\mathbf{P}^*$ respectively satisfying $( u_i,v_j)=\delta_{i,j}$, then the projector $\Pi_0(\lambda)$ on the space of resonant states for the resonance $\lambda$ is given by
\begin{equation}
\label{eq:projo}
\Pi_0(\lambda)=\sum_{i=1}^mu_i\otimes v_i.
\end{equation}
\end{lemm}

In \cite[Theorem 3.10]{Cli}, Climenhaga gives a local product construction of the equilibrium state. For this, let's first notice that Climenhaga's construction is still valid when swapping the unstable and stable foliations.

We can thus define a system of stable leaf measures $\{m^s_x\mid x\in \mathcal M\}$ which is compatible and supported in $\mathcal W^s(x)$. Moreover, $m^s_x$ is defined using formula \eqref{eq:measure}, but using backward Bowen-balls, see \cite[Section 3.3]{Cli} for details. The set $\{m^s_x\mid x\in \mathcal M\}$ is $\varphi_t$-conformal in the sense of \eqref{eq:change of variable} and the change of variable by holonomy is the one explicited in \eqref{eq:COVpi} but with $w^+$ replaced by $w^-$. 

We can then extend the leaf measures to the weak (un)stable foliations
\begin{equation}
\label{eq:weak}
m^{ws}_{x}:=\int_{-\delta}^{\delta}m^s_{\varphi_t x}dt,\quad m^{wu}_{x}:=\int_{-\delta}^{\delta}m^u_{\varphi_t x}dt.
\end{equation}
To state the main result of this subsection, we introduce further notations,
$$ z=\Psi_q(x,y), \quad R_q^u(z)=\mathcal W^u(z,\delta)\cap R_q, \ R_q^{ws}(z)=\mathcal W^{ws}(z,\delta)\cap R_q, $$
where $R_q$ and $\Psi_q$ are defined in equation \eqref{eq:R}. Observe that
\begin{equation}
\label{eq:bracket}
z=[x,y],\quad x=[z,q], \quad y=[q,z].
\end{equation}
We adapt the product construction of Climenhaga (\cite[Theorem 3.10]{Cli}) to obtain the following local construction of the co-resonant state. In the following, we add an additional index $W$ to denote the leaf measures constructed using \eqref{eq:measure} from the Hölder continuous potential $W$. For any $q\in \mathcal M$, we define three measures on $R_q$ by putting, for a Borel set $Z\subset R_q$,
\begin{equation}
\label{eq:1}
m_1^q(Z):=\int_Z e^{w^+_{J^u+V}(z,[z,q])+w^-_{J^u}(z,[q,z])}d((\Psi_q)_*(m^u_{q,J^u+V}\times m^{ws}_{q, J^u}))(z),
\end{equation}
\begin{equation}
\label{eq:2}
m_2^q(Z):=\int_{R^{ws}_q}\int_{Z\cap R^u_q(y)} e^{w^-_{J^u}(z,y)}dm^u_{y,J^u+V}(z) dm^{ws}_{q, J^u}(y),
\end{equation}
\begin{equation}
\label{eq:3}
m_3^q(Z):=\int_{R^{u}_q}\int_{Z\cap R^{ws}_q(x)} e^{w^+_{J^u+V}(z,x)}dm^{ws}_{x,J^u}(z) dm^u_{q,V+J^u}(x).
\end{equation}

\begin{theo}[Construction of the co-resonant state]
\label{Construction}
For any $q\in \mathcal M$, the three measures $m_1^q,m_2^q$ and $m_3^q$ coincide and there is a unique non-zero and finite Borel measure $\nu$ on $\mathcal M$ such that $\forall Z\subset R_q$ one has $\nu(Z)=m_1(Z)$. This Borel measure satisfies 
\begin{equation}
\label{eq:coresonant}
\mathbf{P}^*\nu =P(V+J^u)\nu, \quad \mathrm{WF}(\nu)\subset E_s^*
\end{equation}
and is therefore a co-resonant state. In other words, $P(V+J^u)$ is a Ruelle resonance.
\end{theo}
\begin{rmk}
Climenhaga's original construction is obtained by taking the two potentials to be equal, in other words, all leaf measures are constructed using $V$ and the resulting Borel measure is a non-zero multiple of the equilibrium state $\mu_V$. We observe that when $V=0$, both constructions coincide and we retrieve the known fact  that the SRB measure is the co-resonant state in this case. Of course, by taking the adjoint, we obtain a construction of a resonant state $\eta$.
\end{rmk}
\begin{proof}
We mostly follow the strategy of the proof of \cite[Theorem 10]{Cli}, while only changing the necessary details. We first prove that the three formulas indeed agree and define the same local measure. Starting from a measurable function $\zeta: \mathcal W^u(q,\delta)\times \mathcal W^{ws}(q,\delta)\to ]0,+\infty[$, one sees that \begin{align*}
\int_Z \zeta(\Psi^{-1}_q(z))&d((\Psi_q)_*(m^u_{q,V+J^u}\times m^{ws}_{q, J^u}))(z)
\\&=\int_{\Psi_q^{-1}(Z)}\zeta(x,y) d((m^u_{q,V+J^u}\times m^{ws}_{q, J^u}))(x,y).
\end{align*}
Now, this last expression can be seen to be equal to \eqref{eq:2}, using the change of variable by holonomy \eqref{eq:COVpi} with the holonomy $\pi:\mathcal W^u(q,\delta)\to R^u_q(y), \ \pi(x)=[x,y]$ and with $\zeta(z):=e^{w^+_{V+J^u}(z,[z,q])+w^-_{J^u}(z,[q,z])}$.
\begin{align*}
\int_{R^u_q}\zeta([x,y])\mathbf{1}([x,y])d m^u_{q,V+J^u}(x)=\int_{R^u_q(y)\cap Z}&\zeta(z)d(\pi_*m^u_{q,V+J^u})(z)
\\=\int_{R^u_q(y)\cap Z}\zeta(z)e^{w^+_{V+J^u}([z,q],z)}dm^u_{y,V+J^u}(z)=&\int_{R^u_q(y)\cap Z}e^{w^-_{J^u}(z,[q,z])}dm^u_{y,V+J^u}(z),
\end{align*}
where we used the cocycle relation \eqref{eq:cocycle}. Similarly, we obtain the equivalent formula \eqref{eq:3}. Next, we show that the set of local measures $\{\nu_q\mid q\in \mathcal M\}$ are compatible in the sense that for any Borel set $  Z\subset R_q\cap R_p$, one has $ \nu_q(Z)=\nu_p(Z).$
This will then define a global Borel measure. If $p\in \mathcal W^{ws}(q,\delta)$, then we deduce the above relation follows from \eqref{eq:2} and the fact that the system measures $\{m^s_x\mid x\in \mathcal M\}$ is compatible on the intersection. Similarly, the relation for $p\in \mathcal W^{u}(q,\delta)$ follows from the compatibility of $\{m^{ws}_x\mid x\in \mathcal M\}$ and \eqref{eq:3}. The general case then follows from the local product structure. Call $\nu$ the global Borel measure we obtain. The fact that $\nu$ is non-zero and finite follows from the analog fact on the local measures $\nu_q$ and the compactness of the space $\mathcal M$.
What is left to prove is \eqref{eq:coresonant}. It suffices to show that one has
\begin{equation}
\label{eq:Rela}
\forall f\in C^{\infty}(\mathcal M), \ \langle  e^{t\mathbf{P}}f,\nu\rangle=e^{tP(V+J^u)}\langle f, \nu\rangle,
\end{equation}
which will clearly imply the first part of \eqref{eq:coresonant}. Because the measure $\nu$ is constructed from its restrictions , it will be enough to prove that
\begin{equation}
\label{eq:PPP}
\forall f\in C^{\infty}(\mathcal M), \ \mathrm{Supp}(f)\cup \mathrm{Supp}(e^{t\mathbf{P}}f) \subset R_q \ \Rightarrow \ \nu(e^{t\mathbf{P}}f)=e^{tP(V+J^u)}\nu(f).
\end{equation}
We now use formula \eqref{eq:3} to compute explicitly, using the $\varphi_t$-conformality  \eqref{eq:change of variable}
\begin{align*}
\nu^{V+J^u}(e^{t\mathbf{P}}f)&=\int_{R^{ws}_q}\int_{Z\cap R^{u}_q(y)}e^{S_tV(\varphi_{-t}z)}f(\varphi_{-t}z) e^{w^-_{J^u}(z,y)}dm^{u}_{x, V+J^u}(z) dm^{ws}_{q,J^u}(y).
\end{align*}
The cocycle relation \eqref{eq:cocycle} gives
$e^{w^-_{J^u}(z,y)}=e^{w^-_{J^u}(z,\varphi_{-t}z)}e^{w^-_{J^u}(\varphi_{-t}z,y)}. $
Using \eqref{eq:floww} to get $w^+_{J^u}(z,\varphi_{-t}z)=S_{t}J^u(\varphi_{-t}z)$, this means that 
\begin{align*}\int_{Z\cap R^{u}_q(y)}&e^{S_tV(\varphi_{-t}z)}f(\varphi_{-t}z) e^{w^-_{J^u}(z,y)}dm^{u}_{x, V+J^u}(z) 
\\ =& \int_{Z\cap R^{u}_q(y)}e^{S_t(V+J^u)(\varphi_{-t}z)}f(\varphi_{-t}z) e^{w^-_{J^u}(\varphi_{-t}z,y)}dm^{u}_{x, V+J^u}(z) 
\\ =& \int_{Z\cap R^{u}_q(y)}e^{S_t(V+J^u)(z)}f(z) e^{w^-_{J^u}(z,y)}d\big((\varphi_{-t})_*m^{u}_{x, V+J^u}\big)(z)
\\ =& \int_{Z\cap R^{u}_q(y)}e^{tP(V+J^u)-S_t(V+J^u)(z)}e^{S_t(V+J^u)(z)}f(z) e^{w^-_{J^u}(z,y)}dm^{u}_{\varphi_{-t}x, V+J^u}(z)
\\=& e^{tP(V+J^u)} \int_{Z\cap R^{u}_q(y)}f(z) e^{w^-_{J^u}(z,y)}dm^{u}_{\varphi_{-t}x, V+J^u}(z),
\end{align*}
where we used the $\varphi_t$-conformality, see \eqref{eq:RN}. This implies \eqref{eq:Rela} and thus
$$\mathbf{P}^*\nu =P(V+J^u)\nu. $$
The last thing we need to do is to bound the wavefront set of the measure $\nu$. For this, we will need some regularity result on the conditional measures of the SRB measures and the following adaptation of \cite[Lemma 2.9]{GBGHW}:
\begin{lemm}
\label{Wavefront set}
The wavefront set of $\nu$ is included in $E_s^*$.
\end{lemm}
\begin{proof}
 Consider $q\in \mathcal M$ and  $\xi \notin E_s^*(q)$ and we shall prove that $(q,\xi)$ is not in the wavefront set. Choose a phase function $S$ such that $d_qS=\xi$ and a cutoff $\chi$ near $q$, we then use \eqref{eq:3} to obtain:
$$\nu^{V+J^u}_q(\chi e^{i\frac Sh}):=\int_{R^{u}_q}\int_{Z\cap R^{ws}_q(x)}\!\!\!\!\!\!\!\!\!\!\!\!\chi(z)e^{i\frac Sh} e^{w^+_{V}(z,x)+t(P(V+J^u)-P(V))}e^{w^+_{J^u}(z,x)}dm^{ws}_{x,J^u}(z) dm^u_{q,V+J^u}(x).$$
Now, we can use \cite[Corollary 3.11]{Cli} to obtain that 
$$\frac{d\mu_x^{ws}}{d m^{ws}_{x,J^u}}(z)=\frac{e^{w^+_{J^u}(z,x)}}{m^{ws}_x(R_q^{ws}(x))}, \quad h(x):= m^{ws}_x(R_q^{ws}(x))$$
where $\mu_x^{ws}$ is the conditional measure of the SRB measure on the weak unstable manifold, in other words:
$$\nu^{V+J^u}_q(\chi e^{i\frac Sh}):=\int_{R^{u}_q}h(x)\left(\int_{Z\cap R^{ws}_q(x)}\!\!\!\!\!\!\chi(z)e^{i\frac Sh} e^{w^+_{V}(z,x)+t(P(V+J^u)-P(V))}\mu_x^{ws}(z)\right) dm^u_{q,V+J^u}(x).$$

Now, we use \cite[Theorem 3.9]{Lla} to obtain that the density $\mu_x^{ws}$ is smooth along the leaves $R_q^{ws}(x)$. By this, we mean that 
$$\|\mu_x^{ws}\|_{C^k(R_q^{ws}(x))}:=\sup_{z\in R_q^{ws}(x)}\sup_{X_1,\ldots, X_k \in S_z R_q^{ws}(x)}|X_1\ldots X_k(\mu_x^{ws}(z)_{|R_q^{ws}(x)})| $$
is finite for any $k$ and that the norm depends continuously in $x$.

From the proof of \cite[Corollary 4.4]{Lla}, one sees that the smoothness of the potential $V$ implies that $w_V^+(z,x)$ is smooth along the leaves $R_q^{ws}(x)$. More precisely, the proof shows that given a function which is smooth along the leaves, the function defined by the last integral in \eqref{eq:w+} is smooth along the leaves therefore proving that the holonomy factor is smooth along the leaves.\footnote{Note that the fact that the unstable Jacobian is smooth along the leaves in the sense above is non trivial because it is only Hölder continuous in $x$.}

Using the fact that each leaf is a smooth submanifold, we can perform integration by parts (in $z$) to show that the inner integral is $O(h^{\infty})$ as long as $dS_{|R_w^{ws}}(x)$ does not vanish. But $\xi\notin E_s^*$ so this can be ensured locally near $p$. This proves that $\xi \notin E_s^*$ and we have thus showed that $\mathrm{WF}(\nu)\subset E_s^*$.
\end{proof}
We can now finish the proof of Theorem \ref{Construction}. For this, we use characterization \eqref{eq:resonant} to deduce that $\nu$ is indeed a co-resonant state, i.e $P(V+J^u)$ is a Ruelle resonance.
\end{proof}
\subsection{Critical axis}
In this subsection, we prove that the set of Ruelle resonances is contained in the half plane $\{\lambda \mid \mathrm{Re}(\lambda)\leq P(V+J^u)\}$. The case $V=0$ is enlightening; in this case, the critical axis is the imaginary axis. No resonance has positive real part, indeed, in the case of a volume preserving Anosov flow, this follows directly from Theorem \ref{Resolvent+V} while in the general case, it follows from the general bound $\|e^{-Xt}f\|_{L^{\infty}}\leq \|f\|_{L^{\infty}}$, see \cite[Corr. 4.16]{GBGHW}. We insist on the fact that the functional space $L^{\infty}$ is not really suited for microlocal analysis and thus  this part of the proof needs the introduction of some "finer" measure theory. The goal of this subsection is thus to construct, using the co-resonant state $\nu$ defined in Theorem \ref{Construction}, a norm $\|.\|_V$ such that
\begin{equation}
\label{eq:lips}
\forall f\in C^{\infty}(\mathcal M), \quad \|e^{t\mathbf{P}}f\|_V\leq e^{tP(V+J^u)}\|f\|_{V}.
\end{equation}
\begin{lemm}
\label{new}
Let $\nu$ be as in Theorem \ref{Construction}, for $f\in C^{\infty}(\mathcal M)$, define
$$\|f\|_V:=\|f\|_{L^1(\mathcal M, \nu)}=\int_{\mathcal M}|f|(z)d\nu(z). $$
Then this defines a norm on $C^{\infty}(\mathcal M)$ such that
$$ \forall f\in C^{\infty}(\mathcal M), \quad \|e^{t\mathbf{P}}f\|_V\leq e^{tP(V+J^u)}\|f\|_{V}.$$
We will denote by $L^1(\mathcal M,\nu)$ the completion for $\|.\|_V$ of $C^{\infty}(\mathcal M)$. 
\end{lemm}

\begin{proof}
The fact that $\nu(|f|)\geq 0$ is a consequence of $\nu$ being a measure (i.e of $\nu$ being a distribution which is non-negative on non-negative functions). The homogeneity and the triangle inequality are clear. Suppose now $\|f\|_V=0$, then use the formula \eqref{eq:1} as well as \eqref{eq:measure} to see that the co-resonant state $\nu$ gives positive measure to any non-empty open set of $\mathcal M$. Thus, if $f$ is non-zero, then $|f|$ is positive on a small open set which is a contradiction. The change of variable formula \eqref{eq:lips} is a direct consequence of \eqref{eq:Rela} which was shown in the proof of Theorem \ref{Construction}.
\end{proof}
As a direct consequence, we adapt \cite[Corr. 4.16]{GBGHW} to show that there is no Ruelle resonance  in the half plane $\{\mathrm{Re}(z)> P(V+J^u)\}$.
\begin{lemm}
\label{negative}
The set of Ruelle resonances for the potential $V$ is contained in the half plane $\{\mathrm{Re}(\lambda)\leq P(V+J^u)\}$.
\end{lemm}
\begin{proof}
Recall from the proof of Theorem \ref{Resolvent+V} (see \cite[Theorem 9.11]{Lef}) that we can construct an operator $ \tilde Q(\lambda)$ such that, as an equality on operators acting on $C^{\infty}(M)$, 
\begin{equation}
\label{eq:parametrix}
(\mathbf{P}+\lambda)\tilde Q(\lambda)=\mathrm{Id}-\tilde R(\lambda), 
\end{equation}
where the remainder is given by 
$$\tilde R(\lambda)=-\int_T^{T+\epsilon}\chi'(t)e^{-t\lambda}e^{t \mathbf{P}}dt,$$
where $\chi$ is a cutoff supported in $[0,T+\epsilon[$ and constant equal to $1$ on $[0,T]$, for some suitable choices of $\epsilon$ and $T$.

Let $\lambda\in \mathbb C$ such that $\mathrm{Re}(\lambda)>P(V+J^u)$. Let us show that it is not in the Ruelle spectrum. The projector $\Pi_0(\lambda)$ on the eigenvalue $z=0$ of the Fredholm operator $F(\lambda):=\mathrm{Id}-\tilde R(\lambda)$ is given by the integral
\begin{equation} 
\label{eq:proj}
\Pi_0(\lambda)=\frac{1}{2\pi i}\int_{|z|=\epsilon}(z\mathrm{Id}-F(\lambda))^{-1}dz,
\end{equation}
for a radius $\epsilon>0$ small enough. Note that by the proof of Theorem \ref{Resolvent+V}, this is the spectral projector for any anisotropic Sobolev space $\mathcal H^s$ for which $\mathrm{Re}(\lambda)\geq -cs+C_0$ and thus \eqref{eq:proj} also holds as a map $C^{\infty}(\mathcal M)\to \mathcal D'(\mathcal M)$. But if $f\in L^1(\mathcal M,\nu)$, then, using Lemma \ref{new}, one has
\begin{align*}
\|\tilde R(\lambda)f\|_{V}&\leq \int_T^{T+\epsilon}|\chi'(t)|e^{-t\mathrm{Re}(\lambda)}\times \|e^{t\mathbf{P}}f\|_{V}dt
\\ & \leq \int_T^{T+\epsilon}|\chi'(t)|e^{-t(\mathrm{Re}(\lambda)-P(V+J^u))}\times \|f\|_{V}dt\leq \frac 1 2\|f\|_{V},
\end{align*}
if $T$ is chosen large enough. In particular, this shows that $F(\lambda)$ is invertible on $L^1(\mathcal M,\nu)$ with inverse in $\mathcal L(L^1(\mathcal M,\nu))$ and thus $\Pi_0(\lambda)=0$, meaning that $\lambda$ is not a Ruelle resonance. 
\end{proof}
\subsection{Resonances on the critical axis}
\label{section axis}
In this subsection, we investigate some properties of Ruelle resonances on the critical axis and prove the first part of Theorem \ref{mainTheo}. The main results are: that Ruelle resonances on the critical axis have no Jordan block and that the first resonance is simple.

{The following two lemmas justify that resonances on the critical axis have no Jordan block. The strategy is the same as in \cite[Lemma 5.1]{GBGHW}, where the space $L^{\infty}$ is replaced by the new functional space $L^1(\mathcal M,\nu)$.}
\begin{lemm}
\label{pole}
Let $\lambda+P(V+J^u) \in \{\mathrm{Re}(\lambda)=P(V+J^u)\}$ be on the critical axis. Then $\tilde R(\lambda)$ has spectrum included in the closed unit disk. 
\end{lemm}
\begin{proof}
Consider $B\in \Psi^0(\mathcal M)$ such that its principal symbol is equal to $1$ except in a conic neighborhood $V$ of $E_0^*$. We can then write $\tilde R=\tilde RB+(\mathrm{Id}-B)\tilde R(\mathrm{Id}-B)+B\tilde R(\mathrm{Id}-B),$ where $(\mathrm{Id}-B)\tilde R(\mathrm{Id}-B)$ is smoothing thus compact (because $\mathrm{WF}(\tilde R)$ does not intersect $E_0^*$, see \cite[Theorem 9.11]{Lef} for a detailed computation of the wavefront set). Moreover, we have $\|\tilde RB+B\tilde R(\mathrm{Id}-B)\|_{L^2\to L^2}\leq \frac 1 2$ if $T$ is chosen large enough. This shows that the essential spectrum of $\tilde R$ is contained in $B(0,1/2)$ and that the spectrum outside this ball consists of isolated eigenvalues. Using the $L^2\to L^2$ boundedness of $\tilde R(\lambda)$, we see that for large values of $|z|$, we get the converging Neumann series (in $\mathcal H^s$)
\begin{equation}
\label{eq:Neumann}
(z-\tilde R(\lambda))^{-1}=z^{-1}\sum_{k\geq 0}z^{-k}\tilde R(\lambda)^k. 
\end{equation}
Now, we use the bound \eqref{eq:lips} for $f\in C^{\infty}(\mathcal M)$:
\begin{align*}
\|\tilde R(\lambda)f\|_{L^{1}(\mathcal M,\nu)}\leq \int_T^{T+\epsilon}e^{-tP(V+J^u)}|\chi'(t)|\times \|e^{t\mathbf{P}}f\|_{L^{1}(\mathcal M,\nu)}&dt
\\  \leq \int_T^{T+\epsilon}e^{-tP(V+J^u)}|\chi'(t)|e^{tP(V+J^u)}\|f\|_{L^{1}(\mathcal M,\nu)}dt& \leq \|f\|_{L^{1}(\mathcal M,\nu)},
\end{align*}
to get that the above Neumann series actually converges for $|z|>1$ in $\mathcal L(L^{1}(\mathcal M,\nu))$ and is analytic in $|z|>1$. By density of smooth functions in $\mathcal H^s $, this proves that the eigenvalues are contained in the closed unit disk.

  \end{proof}	
  
The next lemma, together with Remark \ref{JordanBlock}, justifies that resonances on the critical axis have no Jordan blocks.
\begin{lemm}
\label{convergence}
A complex number $\lambda_0\in P(V+J^u)+i\mathbb R$ is a Ruelle resonance, if and only if $\tilde R(\lambda_0)$ has an eigenvalue at $1$. Moreover, if $\tilde R(\lambda_0)$ has a eigenvalue on the unit circle, then it is equal to $1$ and the eigenfunctions correspond to resonant states at the resonance $\lambda_0$. Finally, write $\Pi_0(\lambda_0)$ be the spectral projector of    $\mathrm{Id}-\tilde R(\lambda_0)$ on the kernel $($see equation \eqref{eq:proj}$)$. Then we have the convergence, in $\mathcal L(\mathcal H^s)$, for any $s>0$:
\begin{equation}
\label{eq:convergence}
\Pi_0(\lambda_0)=\lim_{k\to+\infty}\tilde R(\lambda_0)^k.
\end{equation}
\end{lemm}
  \begin{proof}
  Suppose that $\tau \in \mathbb S^1$ is an eigenvalue of $\tilde R(\lambda_0)$. We first check that, for any smooth function $u$, one has $\mathbf{P}\tilde R(\lambda_0)u=\tilde R(\lambda_0)\mathbf{P}u.$
In other words, the operator $\tilde R(\lambda_0)$ commutes with $\mathbf{P}$ so that the space $\mathrm{Ran}(\Pi_{\tau}(\lambda_0))$ associated to the eigenvalue $\tau\in \mathbb S^1$ can be decomposed into generalized eigenspaces of $\mathbf{P}$. Suppose now that $u\in \mathrm{Ran}(\Pi_{\tau}(\lambda_0))$ is an eigenvector for $\mathbf{P}$ for the eigenvalue $\mu$, then we get  
$$\tau u=\tilde R(\lambda_0)u=-\int_0^{\infty}\chi'(t)e^{-t\lambda_0+t\mu}u(x)dt=\widehat{\chi'}(-i\lambda_0+i\mu)u. $$
This implies that $|\widehat{\chi'}(-i\lambda_0+i\mu)|=1$. But using Lemma \ref{negative}, we see that the real part of $\mu$ is smaller than $P(V+J^u)$, so that
$$1=|\tau |= |\widehat{\chi'}(-i\lambda_0-i\mu)|\leq \int_0^{\infty}|\chi'(t)|e^{t\mathrm{Re}(\mu)-tP(V+J^u)}dt\leq 1.$$
This implies that $\mathrm{Re}(\mu)=P(V+J^u)$ and that $\cos(t(-i\lambda_0+i\mu))$ is constant on the support of $\chi'$. This is possible only if $\lambda_0=\mu$ which gives $\tau=1$ and shows that $u$ is a resonant state for the Ruelle resonance $\lambda_0$.

The previous discussion justifies that there exists an integer $N$ such that
$$\Pi_0(\lambda_0):=\{u\in \mathcal D'_{E_u^*}(\mathcal M)\mid (\mathbf P-\lambda_0)^N u=0\}.$$
We will now prove that $\tilde R(\lambda_0)$ has no Jordan block at $1$ and deduce that $N=1$. Note that the right hand side of the previous equality does not depend on the choice of cutoff function and so that needs to be the case for the right hand side. Let $f$ be an eigenfunction of $\tilde R(\lambda_0)$, then
$$ R(\lambda_0)f=\mathrm{Id}=-\int_{\mathbb R}\chi'(t)e^{-t\lambda_0+t\mathbf{P}}fdt=f. $$
 The above relation holds for any cutoff function $\chi$ supported in $[0,T+\epsilon]$ and equal to $1$ on $[0,T]$. The only condition being that $T$ must be large enough to ensure the validity of Lemma \ref{pole}. Approxmating the Dirac mass at $T$ (for $T$ large enough) yields $e^{-T\lambda_0+T\mathbf{P}}f=f$ and $f$ is thus a resonant state at $\lambda_0$. The converse is clear and we showed that the eigenfunctions of $\tilde R(\lambda_0)$ at $1$ coincide with the resonant states at $\lambda_0$. Let $u\in C^{\infty}$ then near $z=1$,
 \begin{equation}
 \label{eq:sauv}(\tilde R(\lambda_0)-z)^{-1}u=\tilde R_H(z)+\sum_{j=1}^{N}\frac{(\tilde R(\lambda_0) -1)^{j-1}\Pi_0(\lambda_0)u}{(z-1)^j}, \end{equation}
where $R_+^H(z)$ is the holomorphic part near $z=1$. We note that $(\tilde R(\lambda_0) -1)^{N-1}\Pi_0(\lambda_0)u$ is an eigenvector of $\tilde R(\lambda)$ for $1$. By the previous discussion, it is thus a resonant state $\theta$ associated to $\lambda_0$. By Lemma \ref{spectral projector}, we can consider the dual co-resonant state $\mu$ such that the pairing $(\theta, \mu)\neq 0$. On one side, using \eqref{eq:sauv} gives, near $z=1$,
\begin{equation}
    \label{eq:eq1}
    \langle (\tilde R(\lambda_0)-z)^{-1}u,\mu\rangle_{\mathcal H^s\times \mathcal H^{-s}}=\frac{(\theta, \mu)}{(z-1)^{N}}+O\big(\frac{1}{(1-z)^{N-1}}\big).
\end{equation}
But on the other hand, for any $|z|>1$, one has, for any $k\geq 0$,
$$z^{-k}\langle \tilde R(\lambda_0)^ku,\mu\rangle=z^{-k}\langle  u,(\tilde R(\lambda_0)^k)^*\mu\rangle=z^{-k}\langle  u,\mu\rangle. $$
Summing the Neumann series then gives
\begin{equation}
    \label{eq:eq2}
    |\langle (\tilde R(\lambda_0)-z)^{-1}u,\mu\rangle_{\mathcal H^s\times \mathcal H^{-s}}|\leq \frac{|\langle u,\mu \rangle |}{|z|^{-1}-1}.
\end{equation}
Going back to \eqref{eq:1}, this is possible only if $N=1$. In particular, resonances on the critical axis have no Jordan block. 
Finally, we use that 
$$\tilde R(\lambda_0)=\Pi_0(\lambda_0)+K(\lambda_0),\quad  K(\lambda_0)\Pi_0(\lambda_0)=\Pi_0(\lambda_0)K(\lambda_0)=0,$$
where the spectral radius of $K(\lambda_0)$ on $\mathcal H^s$ is strictly smaller than $1$. Remember that if $r<1$ is the spectral radius of  $K(\lambda_0)$, then for any $\epsilon>0$ small enough, one has
$$\|K(\lambda_0)^n\|_{\mathcal H^s\to \mathcal H^s}\leq (r+\epsilon)^n\to 0. $$
Thus, as bounded operators on $\mathcal H^s$, we get, as $n\to +\infty$, 
$ \tilde R(\lambda_0)^n=\Pi_0(\lambda_0)+K(\lambda_0)^n\to\Pi_0(\lambda_0).$

\end{proof}
We now identify the co-resonant states associated to resonances on the critical axis to equivariant measures with wavefront set contained in~$E_s^*$ and show that the space of (co)-resonant states at the first resonance is one dimensional and spanned by the measure constructed in Theorem \ref{Construction}. First define for $v\in C^{\infty}(\mathcal M,[0,+\infty[)$ a $($complex$)$ Radon measure $\mu_v^{\lambda}$ such that
\begin{equation}
\label{eq:coreso} \mu_v^{\lambda}:u\in C^{\infty}(\mathcal M)\mapsto \langle \Pi_0(\lambda)u,v\rangle,\ \mathrm{WF}(\mu_v^{\lambda})\subset E_s^*, \ \mu_v^{\lambda}(\mathbf{P}u)=\bar{\lambda} \mu_v^{\lambda}(u). 
\end{equation}
\begin{prop}[Resonant states on the critical axis]
\label{measure on the axis}
Let $\lambda\in P(V+J^u)+i\mathbb R$ be a Ruelle resonance on the critical axis. The space of co-resonant states at $\lambda$ is equal to $\Pi_0^*(\lambda)(C^{\infty}(\mathcal M))=\Pi_0^*(\lambda)(\mathcal H^{-s})$. More precisely, we have the following isomorphism 
$$\Pi_0^*(\lambda)(C^{\infty}(\mathcal M))=\mathrm{Span}\{\zeta_v^{\lambda}, v\in C^{\infty}(\mathcal M; [0,+\infty[)\},$$
Finally, all these measures are absolutely continuous with respect to the measure $\mu_1^{P(V+J^u)}$ with bounded densities.
\end{prop}
 \begin{proof}
 Let $v\in C^{\infty}(\mathcal M,[0,+\infty[)$ and first consider $\lambda=P(V+J^u)$. Then using lemma \ref{convergence}, we get $\tilde R(\lambda)^n\to \Pi_0(\lambda)$ in $\mathcal L(\mathcal H^s)$. Now, for any $u\in C^{\infty}(\mathcal M,[0,+\infty[)$
 \begin{align*}
 \langle \tilde R(P(V+J^u))^ku,v\rangle &=  -\int_{0}^{+\infty}\chi'^{(k)}(t)e^{-tP(V+J^u)}\left(\int_{\mathcal M}e^{-t\mathbf{P}^*}vud\mathrm{vol} \right)dt\geq 0.
 \end{align*}
 This proves that $\langle \Pi_0(P(V+J^u))u,v\rangle\geq 0$ which gives that $\mu_v$ is a non-negative Radon measure for $v\geq 0$. Consider a general resonance on the critical axis $\lambda=ib+P(V+J^u)$ with $b\in \mathbb R$ and $v\leq v'\in C^{\infty}(\mathcal M, [0,+\infty[)$, we get
 $$\forall u\in C^{\infty}(\mathcal M, \mathbb R), \quad |\langle  \tilde R^k(\lambda)u,v\rangle |\leq \langle \tilde R^k(P(V+J^u))|u|,v\rangle \leq \langle \tilde R^k(P(V+J^u))|u|,v'\rangle. $$
 Passing to the limit, we get first that $\mu_v^{\lambda}$ defines an order $0$ distribution and thus defines a (complex) measure. Moreover, the last inequality actually proves that $\mu_v^{\lambda}\ll \mu_{v'}$ with a density bounded by $1$. In particular, one gets that every measure $ \mu_v^{\lambda}$ is absolutely continuous with respect to the measure $\mu_1$.
  \end{proof} 
  
 We can now prove the first part of Theorem  \ref{mainTheo} for the action on $0$-forms.
 \begin{prop}
 \label{hard}
 Under Assumption \ref{assumption}, the first Ruelle resonance $P(V+J^u)$ for the action on $0$-forms is simple with a space of co-resonant state spanned by $\nu$.
 \end{prop}
 \begin{proof}
 Consider a resonant state $\mu$ for the first Ruelle resonance $P(V+J^u)$. We know by the previous proposition that $\mu$ is a measure with wavefront set contained in $E_s^*$ and we will suppose that it is a non negative measure. The wavefront set condition, together with \cite[Corollary 8.2.7]{Hor} justifies that one can restrict the distribution $\mu$ to any unstable manifold $\mathcal W^u(x)$. We will denote by $\mu_x:=\mu_{|\mathcal W^u(x)}$. More explicitly, for a smooth function $f\in C^{\infty}(\mathcal M)$, we can define, for some $s\in \mathbb R$,
 $$(f_{|\mathcal W^u(x)},\mu_x):=(f\delta_{\mathcal W^u(x)}, \mu)=(f\delta_{\mathcal W^u(x)}, \mu)_{\mathcal H^s, \mathcal H^{-s}} $$
 where the bracket denotes the distributional pairing and $\delta_{\mathcal W^u(x)}$ denotes the integration over  the unstable manifold $\mathcal W^u(x)$ defined by
 \begin{equation}
 \label{eq:integr}
 (f, \delta_{\mathcal W^u(x)}):=\int_{\mathcal W^u(x)}f(y)d\mathrm{vol}_{\mathcal W^u(x)}(y).
 \end{equation} The fact that the distributional pairing coincides with the $\mathcal H^s \times \mathcal H^{-s}$ pairing is a consequence of the wavefront set condition on the distributions. In particular, $\mu_x$ is a non negative distribution as the product of two non negative distributions. The restrictions $\{\mu_x \mid x\in \mathcal M\}$ define a system of $\mathcal W^{ws}$-transversal measures in the sense of \cite{Cli}. The strategy to show that the first resonance is simple is to use \cite[Corollary 3.12]{Cli} which asserts that the only system (up to constant rescaling) of $\mathcal W^{ws}$-transversal measures which satisfies $\varphi_t$-conformality (see \eqref{eq:change of variable}) and the change of variable by holonomy (see \eqref{eq:COVpi}) is $m^u_{V+J^u}$.
 
 \textbf{The system of measures $\{\mu_x \mid x\in \mathcal M\}$ satisfies $\varphi_t$-conformality.}
 This is a consequence of the fact that $\mu$ is a co-resonant state. More precisely, for any smooth function $f\in C^{\infty}(\mathcal M)$, we can write
 $$( f\delta_{\mathcal W^u(x)}, \mu)_{\mathcal H^s \times \mathcal H^{-s}}=e^{-tP(V+J^u)}( e^{t\mathbf P}( f\delta_{\mathcal W^u(x)}),\mu)_{\mathcal H^s \times \mathcal H^{-s}}.$$
Using $e^{t\mathbf P}( f\delta_{\mathcal W^u(x)})=e^{t\mathbf P} fe^{-tX}\delta_{\mathcal W^u(x)}$, then yields
$$\int_{\mathcal W^u(x)}f(y)d\mu_x(y)= \int_{\mathcal W^u(\varphi_t x)}e^{-tP(V+J^u)}e^{S_tV(\varphi_{-t}y)}f(\varphi_{-t}y) \det (d\varphi_{-t})_{|E_u(y)}d\mu_{\varphi_t x}(y).$$
This rewrites 
$$\int_{\mathcal W^u(x)}f(y)d\mu_x(y)= \int_{\mathcal W^u(\varphi_t x)}e^{-tP(V+J_u)}e^{S_t(V+J^u)(\varphi_{-t}y)}f(\varphi_{-t}y)d\mu_{\varphi_t x}(y)$$
which is exactly $\varphi_t$-conformality for the potential $V+J^u$.

\textbf{The system of measures $\{\mu_x \mid x\in \mathcal M\}$ satisfies the change of variable by holonomy.}
Investigating the proof of \cite[Corollary 3.12]{Cli}, we see that we only need to show the change of variable formula for the standard $\delta_0$-holonomy given by $\mathcal W^u(x)\ni z\mapsto \pi(z):=[z,x']\in \mathcal W^u(x')$. 
We have the following facts:
\begin{itemize}
\item For all $x,x'\in \mathcal M$, $x\neq x'$ close, one has $L_{x,x'}:\phi\in C^{0}(\mathcal M)\mapsto \frac{(\mu_x-\mu_{x'},\phi)}{d(x,x')^{\alpha}}$ is linear.
\item For any $\phi\in C^{0}(\mathcal M)$, one has $\sup_{x\neq x'} |L_{x,x'}(\phi)|<+\infty$ by \eqref{eq:appendix}.
\item The space $C^{0}(\mathcal M)$ is a Banach space.
\end{itemize}
We can apply the Banach Steinhauss theorem to get
$\sup_{x\neq x'} \|L_{x,x'}\|_{\mathrm{op}}=C<+\infty $, i.e
$$ \forall \phi\in C^{0}(\mathcal M), \forall x,x'\in \mathcal M, \ |(\mu_x-\mu_{x'},\varphi)|\leq Cd(x,x')^{\alpha}.  $$
Now consider $\mathcal W^u(x)\ni z\mapsto \pi(z):=[z,x']\in \mathcal W^u(x')$. Our goal is to relate $\mu_x$ and $\pi_{x'}^*\mu_{x'}$. Consider $f\in C^0(\mathcal W^u(x,\delta_0))$. Since the foliation is Hölder, we can extend $f$ locally on a small rectangle by making it constant on weak-stable manifolds and this defines a continuous function $\bar f$. We are now left with relating $\mu_x(\bar f)$ and $\mu_{x'}(\bar f)$.

 First consider the Bowen time $t=t(x,x')$ such that $\varphi_tx\in \mathcal W^s(x')$. Then $\varphi_t$-conformality shows that $\tfrac{d[(\varphi_t)_*\mu_x]}{d\mu_{\varphi_tx}}(z)=e^{S_t(V+J^u)(z)-tP(V+J^u)}$. Now, for any $\tau>0$, we can use $\varphi_t$-conformality again to get $\tfrac{d[(\varphi_\tau)_*\mu_{\varphi_tx}]}{d \mu_{\varphi_{t+\tau}x}}(z)=e^{S_{\tau}(V+J^u)(z)-\tau P(V+J^u)}$ and $\tfrac{d[(\varphi_\tau)_*\mu_{x'}]}{d\mu_{\varphi_\tau x'}}(y)=e^{S_{\tau}(V+J^u)(y)-\tau P(V+J^u)} $. Now, we have that $d(\varphi_{t+\tau}x, \varphi_\tau x')\leq Ce^{-\eta \tau}$ and the above continuity then shows that $|\mu_{\varphi_{t+\tau}x}( \bar f)-\mu_{\varphi_\tau x'}(\bar f)|\to 0$ when $\tau\to +\infty$. This means that the overall holonomy factor is given by the following limit  
 $$\exp\left(\lim_{\tau \to +\infty} S_t(V+J^u)(z)-tP(V+J^u)+S_\tau(V+J^u)(\varphi_{\tau+t} z)-S_{\tau}(V+J^u)(\varphi_{\tau} y)\right) $$
 which is exactly equal to $e^{w_{V+J^u}^+(z,y)}$, see the proof of \cite[Theorem 3.9]{Cli} for a more detailed argument.
 
 \textbf{The first resonance $P(V+J^u)$ is simple.}
 We can use \cite[Corollary 3.12]{Cli} to claim that there is a constant $c>0$ such that for any $x\in \mathcal M$, one has $\mu_x=c m^u_{x,V+J^u}$. We now would like to deduce that this implies $\mu=c\nu$ by proving a Fubini like formula for distributions with wavefront set in $E_s^*$. For this, we first need to desintegrate the volume measure with respect to the unstable foliation and an arbitrary (smooth) transversal complementary foliation $G$. More precisely, we will use \cite[Proposition 1.6]{GBGW} which states that in a product neighborhood $U$ around a point $q\in \mathcal M$, there is a continuous function $\delta_u: U\to \mathbb R_+$ such that
 $$\forall f\in C^{\infty}(\mathcal M), \ \int_U fd\mathrm{vol}=\int_{G_{\mathrm{loc}}(p)}\left(\int_{\mathcal W^u(y)}f(z)\delta_u(z)d\mathrm{vol}_{\mathcal W^u(y)}(z)\right) d\mathrm{vol}_p^G(y) $$
 where $G_{\mathrm{loc}}(p)$ is the connected component of the element of the partition of $G$ containing $p$.
 Actually, the \emph{conditional density} $\delta_u$ is smooth along the leaves in the following sense:
 $$\|\delta_u\|_{C^k(R_q^{ws}(x))}:=\sup_{z\in \mathcal W^{u}(x)}\sup_{X_1,\ldots, X_k \in S_z \mathcal W^{u}(x)} |X_1\ldots X_k(\delta_u(z))_{|\mathcal W^{u}(x)}| $$
is finite for any $k$ and that the norm depends continuously in $x$. This regularity condition allows to integrate by parts in the inner integral and in particular, an immediate adaptation of \cite[Lemma 1.9]{GBGW} yields $\mathrm{WF}(\delta_u(z) \delta_{\mathcal W^u(y)})\subset E_u^*\oplus E_0^*$. Here, we have denoted by $\delta_u(z) \delta_{\mathcal W^u(y)}$ the distribution:
$$f\mapsto  \int_{\mathcal W^u(y)}f(z)\delta_u(z)d\mathrm{vol}_{\mathcal W^u(y)}(z).$$
 In particular, is we consider $f_n\in C^{\infty}(\mathcal M)$ such that $f_n\to \mu$ is $\mathcal D'_{\Gamma}(\mathcal M)$ where $\Gamma$ is a conic neighborhood of $E_s^*$ which does not intersect $E_u^*\oplus E_0^*$, we can use the continuity of the distributional product (see \cite[Lemma 4.2.7]{Lef}) to write for any $f\in C^{\infty}(\mathcal M)$,
\begin{align*}(f,\mu)=&\lim_{n\to +\infty}(f,f_n)= \lim_{n\to +\infty}\int_{G_{\mathrm{loc}}(p)} (\delta_u(z) \delta_{\mathcal W^u(y)},f\times f_n)d\mathrm{vol}_p^G(y)
\\&=\int_{G_{\mathrm{loc}}(p)} (\delta_u(z) \delta_{\mathcal W^u(y)},f\times \mu)d\mathrm{vol}_p^G(y)
\\&= \int_{G_{\mathrm{loc}}(p)} \mu_y(f\times \delta_u)d\mathrm{vol}_p^G(y)=c\int_{G_{\mathrm{loc}}(p)} m^u_{y,V+J^u}(f\times \delta_u)d\mathrm{vol}_p^G(y)
\\&=c(f,\eta),
\end{align*}
 where we used the continuity of the density and the fact that $\mu_x$ is of order $0$. Together with the previous lemma, this proves that all co-resonant states are proportional and thus that $P(V+J^u)$ is simple and concludes the proof.
 \end{proof}   
We can now prove Theorem \ref{equilibriumtheo} for the action on $0$-forms.
\begin{prop}
\label{nothard}
The equilibrium state $($see \eqref{eq:variational}$)$ $\mu_{V+J^u}$ is equal to $=c\eta \times \nu$ or is given by the averaging formula \eqref{eq:average}. 
Finally, if the flow is weak mixing with respect to the equilibrium state $\mu_{V+J^u}$, then the only resonance on the critical axis is $P(V+J^u)$.
\end{prop}

\begin{proof}
\textbf{The co-resonant state $\nu$ is absolutely continuous with respect to $\mu_{V+J^u}$.}
We follow the proof of \cite[Theorem 3.10]{Cli} and prove that there exists a constant $C>0$ independent of $q\in \mathcal M$ and $t\geq 0$ such that
\begin{equation}
\label{eq:ab}
\forall t\geq 0, \ \forall q\in \mathcal M, \quad \nu(B_t(q,r))\leq Ce^{S_t(V+J^u)(q)-tP(V+J^u)} 
\end{equation} which suffices as $\mu_{V+J^u}$ is the the Gibbs state (see \cite[Theorem 4.3.26]{FishHas}). From the proof of \cite[Theorem 3.10]{Cli} we get that
$$\exists C>0, \forall q\in \mathcal M, \forall t\geq 0, \quad m^{u}_{V+J^{u}}(B_t(q,\delta))\leq C e^{S_t(V+J^u)(q)-tP(V+J^u)}. $$
Now, we can use equation \eqref{eq:3}, the fact that the integrand $e^{w_{J^u}^-(z,x)}$ is uniformly bounded on $\mathcal M$ and $m^{ws}_{q,J^u}(R_q^{ws})$ is uniformly bounded in $q$ (see \cite[Theorem 3.1]{Cli}) to get \eqref{eq:ab}. 

\textbf{Product formula and averaging formula.}
Let  $\eta$ be the resonant state, which we can obtain by applying Theorem \ref{Construction} to the dual $\mathbf{P}^*$. We have $\mathrm{WF}(\eta) \subset E_u^*$, $\mathrm{WF}(\nu) \subset E_s^*$ and we have $E_s^*\cap E_u^* \cap (T^*\mathcal M\setminus \{0\}) = \emptyset$ which shows that the distributional product of $\eta$ and $\nu$ is well defined. Actually, we can use the duality \eqref{eq:dual2}  to define the product:
\begin{equation}
\label{eq:def}
\forall f\in C^{\infty}(\mathcal M),\ (\eta \times \nu)(f):=(\eta, f \nu)_{\mathcal H^s \times \mathcal H^{-s}}=(f\eta, \nu)_{\mathcal H^s \times \mathcal H^{-s}}
\end{equation}
where the equality is justified by \eqref{eq:compute}. From the fact that $\nu$ and $\eta$ are non-negative measures (see Proposition \ref{measure on the axis}) and \eqref{eq:def}, we see that this is also the case of their product $\eta\times\nu$. It is easily seen to be invariant by the flow:
\begin{align*}
 (\eta \times \nu)(Xf)&=( X^*\eta,f\nu)-(f\eta,X\nu)
 \\&=P(V+J^u)(\eta,f\nu)-(V\eta,f\nu)+(\mathrm{div}_{\mathrm{vol}}(X)\eta,f\nu)
 \\&-P(V+J^u)(f\eta,\nu)+(f\eta,V\nu)-(\mathrm{div}_{\mathrm{vol}}(X)f \eta,\nu)=0.
\end{align*}
We prove that the product is absolutely continuous with respect to $\mu_{V+J^u}$. For this, we use Proposition \ref{measure on the axis} and the previous proposition to get $\Pi_0(P(V+J^u))1=c\eta$ for some $c>0$. Now, this means that we can compute $(\eta \times \nu)(f)$ for $f\in C^{\infty}(\mathcal M)$:\begin{align*}
c(\eta \times \nu)(f)&=( \Pi_0(P(V+J^u))1,f\nu)_{\mathcal H^s \times \mathcal H^{-s}} 
\\&=-\lim_{k\to +\infty}\int_0^{+\infty}(\chi')^{(k)}(t)e^{-tP(V+J^u)}( 1, e^{t\mathbf{P}^*}(f\nu)) dt
\\&=-\lim_{k\to +\infty}\int_0^{+\infty}(\chi')^{(k)}(t)( e^{tX}f, \nu) dt.
\end{align*}
We use the fact that $\nu$ is absolutely continuous with respect to $\mu_{V+J^u}$ with bounded density to get that $|\nu(e^{tX}f)|\leq C\mu_{V+J^u}(|e^{tX}f|)\leq C\mu_{V+J^u}(|f|)$ because $\mu_{V+J^u}$ is invariant. In particular, we get, by integrating:
$$\forall f\in C^{\infty}(\mathcal M),\quad  (\eta \times \nu)(|f|)\leq C'\mu_{V+J^u}(|f|) \ \Rightarrow \ (\eta \times \nu)\ll \mu_{V+J^u}. $$
This proves that the density of $(\mu \times \nu)$ is invariant by the flow and measurable for the equilibrium state $\mu_{V+J^u}$. But this last measure is known to be ergodic so the density is constant, which proves the two formulas.

\textbf{Weak mixing implies no other resonances on the critical axis.}
The fact that for a transitive Anosov flow, topological mixing is equivalent to not being a constant time suspension of an Anosov diffeomorphism is known as the Anosov alternative and can be found in \cite[Theorem 8.1.3]{FishHas}. Suppose now that the flow is topologically mixing. This is equivalent to being weakly mixing with respect to any equilibrium state, indeed, topological mixing implies weak mixing by a result of Bowen \cite[Theorem 7.3.6]{FishHas} and constant time suspensions are not weakly mixing for any equilibrium state. 

We first show that weak-mixing with respect to $\mu_{V+J^u}$ implies that there is no other resonance on $\mathcal C_0.$
From the averaging formula \eqref{eq:average}, we see that if $\nu$ gives zero measure to an invariant Borel set, then it is also the case for $\mu_{V+J^u}$ and the two measures are actually equivalent.

Next, we use \cite[Theorem VII.14]{ReSi} to see that the flow is weakly mixing if and only if the only eigenvalue is $1$ and it is a simple eigenvalue. In other words, if 
\begin{equation}
\label{eq:thrid2}
\begin{cases}
Xf=i\lambda f
\\ f\in L^2(\mathcal M,\mu_{V+J^u})
\end{cases}
\end{equation}
has no solution except for $\lambda=0$ and $f$ constant. Using Proposition \ref{measure on the axis}, the presence of another co-resonant state on the critical axis is equivalent to the existence of an absolutely continuous measure with respect to $\nu$. Its density with respect to $\nu$ is in $L^{\infty}(\mathcal M,\nu)=L^{\infty}(\mathcal M,\mu_{V+J^u})\subset L^2(\mathcal M,\mu_{V+J^u})$ and is thus a solution of the  system above. This shows that weak mixing of the flow with respect to $\mu_{V+J^u}$ implies that there is no other resonance on the critical axis. 

Conversely, suppose that the flow is not weak mixing with respect to $\mu_{V+J^u}$, i.e that \eqref{eq:thrid2} admits an $L^2$ solution. We can use \cite[Theorem 4.3]{Wal} for the Lie group $G=S^1$ which gives that the solution $h$ of \eqref{eq:thrid2} of  can be chosen to be smooth. But then, we see that $\gamma:=h\nu$ satisfies
$$\mathbf{P}^*\gamma=(P(V+J^u)+i\lambda)\gamma, \quad \mathrm{WF}(\gamma)\subset E_s^*, $$
that is, $P(V+J^u)+i\lambda$ is a resonance on $\mathcal C_0$.

\end{proof}
\section{Acting on the bundle of $d_s$-forms} 
\label{sec5}
We have studied the case of $0$-forms before and for this, we have fixed a Riemmanian volume $\mathrm{vol}$ to embed smooth functions into distributions. As a consequence, the co-resonant state $\nu$ from Theorem \ref{Construction} has a rather convoluted form. Note however that changing the metric does not change the critical axis as the pression of two cohomologic functions are the same. For $d_s$-forms, such an identification is not possible as the bundles involved are not line bundles anymore and we are led to adopt the more general and intrisic viewpoint of \emph{currents}. Loosely speaking, in this new formalism, a distribution is an element of the topological dual of $0$-forms. More generally, we will think of "distributions" on $d_s$-forms as continuous linear forms on $C^{\infty}(\mathcal M, \Lambda^{d_s}T^*\mathcal M)$ and will call this a $d_s$-\emph{current}. The construction of anisotropic Sobolev space is also valid\footnote{Smoothness is needed to apply the microlocal approach. Here, the bundles $E_s^*$ and $E_u^*$ are only Hölder continuous but $\mathscr E_0^k$ is smooth.} on this space, as noted before in Remark \ref{remm}. We will focus our attention on forms in the kernel of the contraction by the flow, they are given by
$$\mathscr E_0^k:=\{ u\in C^{\infty}(\mathcal M; \Lambda^k T^*\mathcal M) \mid \iota_{X}u=0\} =C^{\infty}(\mathcal M;\Lambda^k(E_u^*\oplus E_s^*)).$$
\subsection{Critical axis}
 
 More generally, the action of $\mathbf{P}$ will be extended, by duality, to \emph{currents}. This idea can be traced back at least to Ruelle and Sullivan although no spectral theory was involved in \cite{RuelleDavid1975Cfad} . We define a $k$-current $T$ as a continuous linear form on the space of smooth $k$-forms. On manifolds, especially when no canonical choice of smooth volume form is available, currents allow us to obtain an intrisic  definition of distributions. As such, the space of smooth $k$-forms is embedded canonically (up to choosing an orientation) in the space of $(n-k)$-currents by the formula
$$\forall \varphi \in C^{\infty}(\mathcal M;\Lambda^{k}T^*\mathcal M), \ \varphi : \alpha\in C^{\infty}(\mathcal M;\Lambda^{n-k}T^*\mathcal M) \mapsto \int_{\mathcal M} \alpha \wedge \varphi. $$ 
Note that the set of smooth $k$-forms is dense in the space of homogeneous currents of degree $n-k$, see \cite[Theorem 12 Section 15]{DeR} for a precise statement and more generally \cite[Chapter 3]{DeR} for an introduction to currents. In the rest of the paper, we will write $\mathcal D' (\mathcal M; \Lambda^{d_u}(E_u^*\oplus E_s^*))$ for the space of sections of currents of degree $d_s$ which are cancelled by the contraction $\iota_X$. They can be thought as linear combinations of elements of $\Lambda^{d_u}(E_u^*\oplus E_s^*)$ with distributional coefficients. We first introduce a useful (Hölder continuous) splitting;
\begin{equation}
\label{eq:split}
\Lambda^{q}(E_u^*\oplus E_s^*)=\bigoplus_{k=0}^{q}\big(\Lambda^kE_s^*\wedge \Lambda^{q-k}E_u^*\big)=:\bigoplus_{k=0}^{q} \Lambda_k^{q}.
\end{equation}
We prove that the system of (unstable) leaf measures $\{m^u_x\mid  x\in \mathcal M\}$ defines a section of currents of Sobolev regularity $\mathcal H^s (\mathcal M; \Lambda^{d_s}(E_u^*\oplus E_s^*))$ and is actually a co-resonant state for $\lambda=P(V)$. This point reduces to a computation of the wavefront set of the system of leaf measures $m^u_V$ by \eqref{eq:resonant}. This is done by adapting the argument of Lemma \ref{Wavefront set} and we note that the proof is actually easier in this case as we do not have to deal with the smoothness of the unstable Jacobian $J^u$ along the leaves.  We will then use the "$L^1$-norm" associated to $m^u_V$ to prove that no Ruelle resonances exist in the half plane $\{\mathrm{Re}(\lambda)>P(V)\}$ by mimicking the proof of Lemma \ref{negative}. 

However, because we work on forms, the co-resonant state $m^u_V$ will be non zero only when tested on sections with values in $\Lambda_0^{d_s}$ and will fail to define a norm on the rest of the decomposition. This will explain the need of a more complicated norm, which we will construct in Proposition \ref{realnorm}.

More importantly, the decomposition \eqref{eq:splitt} is only Hölder continuous which prevents us from applying the microlocal strategy to obtain a meromorphic extension of the resolvent acting on sections with values in $\Lambda_k^{d_s}$. 

Indeed, in general, one can only expect to have Hölder continuous section  $C^{\alpha}(\mathcal M: \Lambda_k^{d_s})$ for some $\alpha>0$. The pairing against $m^u_V$ will however still be justified because the co-resonant state is of order $0$.

\begin{prop}
\label{muv}
The system of leaf measures $\{m^u_x\mid  x\in \mathcal M\}$ from \eqref{eq:measure} defines a section of $\mathcal D' (\mathcal M; \Lambda^{d_u}(E_u^*\oplus E_s^*))$ which we will denote by $m^u_V$. Moreover, one has
\begin{equation}
\label{eq:eee}
\mathcal L_{\mathbf{P}^*}m^u_V=P(V)m^u_V, \quad \mathrm{WF}(m^u_V)\subset E_s^*, 
\end{equation}
thus $m^u_V\in \mathcal H^s (\mathcal M; \Lambda^{d_u}(E_u^*\oplus E_s^*))$ and it is a co-resonant state associated to the Ruelle resonance $P(V)$. We will call $P(V)$ the first resonance.

Moreover, for any $1\leq k\leq d_s$ and $\omega_k \in C^{\alpha}(\mathcal M; \Lambda_k^{d_s})$, one has $m^u_V(\omega_k)=0$.
\end{prop}
\begin{proof}
Our first goal is to give meaning to $m^u_V(\varphi)$ for $\varphi\in C^{\infty}(\mathcal M; \Lambda^{d_s}T^*\mathcal M)$. First, the compatibility statement allows us to only define the duality locally, so let $\varphi\in C^{\infty}(\mathcal M; \Lambda^{d_s}T^*\mathcal M)$ be supported in $R_q$. We can define the duality as follows:
\begin{equation}
\label{eq:muv}
m^u_V(\varphi):=\int_{\mathcal M}m^u_V\wedge \alpha \wedge \varphi=\int_{\mathcal W^u(q,\delta)}\left( \int_{R_q^{ws}(x)}e^{w_V^+(y,x)}(\varphi\wedge \alpha)(y)\right) dm^u_V(x),
\end{equation}
where $\ \alpha\in E_0^*, \ \alpha(X)=1$.
We see that the previous definition makes sense as $\varphi\wedge \alpha$ is a $d_s+1$ form and as $w_V^+(x,y)$ is smooth in $y$, in the sense of Lemma \ref{Wavefront set}, and continuous in $x$. The formula clearly defines a current of degree $0$ so $m^u_V\in D' (\mathcal M; \Lambda^{d_u}(E_u^*\oplus E_s^*))$.

Let $\omega_k \in C^{\alpha}(\mathcal M: \Lambda_k^{d_s})$ for some $k$. Then the fact that $m^u_V$ is a measure on $\mathcal W^u(q,\delta)$ shows that the pairing $m^u_V(\omega_k)$ is well defined. If $k\geq 1$, we can use \eqref{eq:muv} and the definition of $E_s^*$ to get that $m^u_V(\omega_k)=0$.

To prove the first part of \eqref{eq:eee}, we can consider a $d_s$ form $\varphi$ supported in $R_q$ such that $e^{-t\mathbf{P}}\varphi$ is also supported in $R_q$. We have to prove that  $m^u_V(e^{t\mathbf{P}}\varphi)=e^{tP(V)}m^u_V(\varphi)$. We first use the Leibniz rule to get 
$\mathcal L_X(\varphi \wedge \alpha)=\mathcal L_X\varphi \wedge \alpha+\varphi \wedge \mathcal L_X\alpha=\mathcal L_X\varphi \wedge \alpha. $
This allows us to compute explictly the action of the propagator on the current $m^u_V$.
 More precisely, we use  the cocycle relation \eqref{eq:cocycle}, equation \eqref{eq:floww} as well as $\varphi_t$ conformality  \eqref{eq:change of variable} :
\begin{align*}
m^u_V(e^{t\mathbf{P}}\varphi)&=\int_{\mathcal W^u(q,\delta)}\left( \int_{R_q^{ws}(x)}e^{w_V^+(y,x)}e^{S_tV(\varphi_{-t}y)}e^{t\mathcal L_{{X}}}\big(\varphi \wedge \alpha\big)(y)\right) dm^u_V(x)
\\&=\int_{\mathcal W^u(q,\delta)}\left( \int_{R_q^{ws}(\varphi_tx )}e^{w_V^+(\varphi_t w,x)}e^{S_tV(w)}(\varphi \wedge \alpha)(w)\right) dm^u_V(x)
\\&=\int_{\mathcal W^u(q,\delta)}e^{w_V^+(\varphi_t x,x)}\left( \int_{R_q^{ws}(\varphi_t x )}e^{w_V^+(\varphi_t w,w)}e^{w_V^+(w,\varphi_tx)}e^{S_tV(w)}(\varphi \wedge \alpha)(w)\right) dm^u_V(x)
\\&=\int_{\mathcal W^u(q,\delta)}e^{2tP(V)}e^{-S_tV(x)}\left( \int_{R_q^{ws}(\varphi_t x )}e^{w_V^+(w,\varphi_tx)}(\varphi \wedge \alpha)(w)\right) dm^u_V(x)
\\&=e^{tP(V)}\int_{\mathcal W^u(q,\delta)}\left(\int_{R_q^{ws}(z )} e^{w_V^+(z,w)} (\varphi \wedge \alpha)(z)\right)dm^u_V(z)=e^{tP(V)}m^u_V(\varphi).
\end{align*}
For the wavefront set condition, we mimick the argument of Lemma \ref{Wavefront set} and consider a smooth $d_s$-form $\chi$ supported in $R_q$ and $S$ a phase function such that $dS(q)=\xi \notin E_s^*$ and compute
$$ m^u_V(e^{i\frac S h}\chi)=\int_{\mathcal W^u(q,\delta)}\left( \int_{R_q^{ws}(x)}e^{w_V^+(y,x)}e^{i\frac{S(y)} h}(\chi\wedge \alpha)(y)\right) dm^u_V(x).$$
Now, the proof is easier than for Lemma  \ref{Wavefront set} as the integrand is easily seen to be smooth along the weak-stable leaves (because the potential $V$ is smooth) uniformly in $x$. We can perform integration by parts and show that the integrand is $O(h^{\infty})$ as long as $dS$ does not vanish on $R_q^{ws}(x)$, which can be ensured near $q$ by the definition of $E_s^*$. This shows that $\xi \notin \mathrm{WF}(m^u_V)$ and thus $ \mathrm{WF}(m^u_V)\subset E_s^*$. This shows that $P(V)$ is a Ruelle resonance with the associated Ruelle co-resonant state given by $m^u_V$.
\end{proof} 
If we consider the adjoint $\mathbf{P}^*$, the above theorem would give that $m^s_V$ is a resonant state for $P(V)$. We also get a product construction of the equilibrium measure. In this case, it is actually easier than in the case of functions as the product is given by the usual wedge product (extended to distributions with convenient wavefront sets). The following lemma is a re-writing of \cite[Theorem 3.10]{Cli} and corresponds to the second part of Theorem \ref{equilibriumtheo}.
\begin{lemm}[Equilibrium state as a product] There exists $c>0$ such that
\label{lemmm}
$$ m^u_V\wedge \alpha \wedge m^s_V=c\mu_V.$$
\end{lemm}
 
Similarly to the case of functions, we now prove that $\{\mathrm{Re}(\lambda)=P(V)\}$ is the critical axis by using the "$L^1$-norm" associated to the co-resonant state $m^u_V$. Nevertheless, because we work on forms, the following norm will only be well defined for section with values in $ \Lambda_0^{d_s}$.
\begin{lemm}
\label{Crit}
We define a norm on $C^{0}(\mathcal M; \Lambda_0^{d_s})$ by posing
\begin{equation}
\label{eq:newnorm2}
\forall \varphi \in  C^{0}(\mathcal M; \Lambda_0^{d_s}), \ \|\varphi\|_{V,0}:= m^u_V(|\varphi|).
\end{equation}
This norm satisfies the bound
\begin{equation}
\label{eq:bbound}
\forall  \varphi \in  C^{0}(\mathcal M; \Lambda_0^{d_s}), \ \|e^{t\mathbf{P}}\varphi\|_{V,0}\leq e^{tP(V)}\|\varphi\|_{V,0}.
\end{equation}

\end{lemm}
\begin{proof}
Suppose $\varphi \in  C^{0}(\mathcal M;\Lambda^{d_s}_0)$, the bundle is one dimensional and it thus makes sense to talk of $|\varphi|\wedge \alpha$ as a $d_s+1$ density. If $\varphi(q)\neq 0$ then by continuity, $\varphi\neq 0$ on a small open set. Then $\|\varphi \|_{V,0}>0$ because $m^u_V$ gives a positive measure to any open set. The bound \eqref{eq:bbound} follows from the fact that $m^u _V$ is a co-resonant state. The last point follows from the first two points by a direct adaptation of the proof of Lemma \ref{negative}.
\end{proof}

We now use this norm and a "shift" to define inductively a norm on $C^0(\mathcal M; \Lambda_k^{d_s})$ for any $k$.
Consider the set of all finite covers by open sets which trivialize $E_s$ and $E_u$:
$$\mathcal C:=\{\mathcal U:=(U_j)_{1\leq j \leq n}\mid  \mathcal M=\cup_{j=1}^n U_j, \ U_j \text{ open and } E_s \text{ and } E_u \text{ are trivial on } U_j\}. $$
For any $\mathcal U\in \mathcal C$, let $\mathcal P(\mathcal U)$ be the set of partition of unity $\chi_j$ associated to the cover $\mathcal U$. Finally, we define the set of (normalized) local trivialization of $E_u$:
$$\mathscr V^u(\mathcal U):=\{ (X^j_{u,h})_{1\leq j \leq d_u}\in C^0(U_h; E_u) \mid (E_u)_{|U_h}=\mathrm{Span}\{X^j_{u,h}\}_{1\leq j \leq d_u}, \ \| X^j_{u,h}\|_{C^0}=1\},$$ 
and its dual conterpart
$$\mathscr F^u(\mathcal U):=\{ (Y^j_{u,h})_{1\leq j \leq d_s}\in C^0(U_h; E_u^*) \mid (E_u^*)_{|U_h}=\mathrm{Span}\{Y^j_{u,h}\}_{1\leq j \leq d_s}, \ \| Y^j_{u,h}\|_{C^0}=1\}.$$ 
\begin{prop}
\label{realnorm}
We define a norm inductively on $C^0(\mathcal M; \Lambda_{k}^{d_s}), k\geq 1$, by posing, for $f \in  C^{0}(\mathcal M; \Lambda_k^{d_s})$
\begin{equation}
\label{eq:T}
   \|f\|_{V,k}:= \sup_{\mathcal U\in \mathcal C}\sup_{(\chi_j)\in \mathcal P(\mathcal U)}\max_{h=1}^n \sup_{(X^j_{u,h})\in \mathscr V^u}\sup_{(Y^i_{u,h})\in \mathscr F^u}\sum_{i=1}^{d_s}\sum_{j=1}^{d_u}\|\chi_h\iota_{X_{u,h}^j}f\wedge Y_{u,h}^{i}\|_{V,k-1}.
\end{equation}
This norm satisfies the bound
\begin{equation}
\label{eq:bbbound}
\forall  \varphi \in  C^{0}(\mathcal M; \Lambda_k^{d_s}), \ \|e^{t\mathbf{P}}\varphi\|_{V,k}\leq C e^{t(P(V)-k\eta)}\|\varphi\|_{V,k}
\end{equation}
for some $C,\eta>0$. As a consequence, there are no Ruelle resonances in $\{\mathrm{Re}(\lambda)>P(V)\}$ and if $\lambda$ is a Ruelle resonance on the critical axis, then it has no Jordan block.
\end{prop} 
\begin{proof}
Contracting with a vector in $E_u$ and then wedging with a vector in $E_u^*$ sends a section of $C^0(\mathcal M; \Lambda_{k}^{d_s})$ to $C^0(\mathcal M; \Lambda_{k-1}^{d_s})$ so the formula makes sense by induction. Let's prove that the norm takes finite values. For this, we use the fact that $m^u_V$ is of order zero and Lemma \ref{Crit} to get first that for $f\in C^0(\mathcal M; \Lambda_{0}^{d_s})$, one has $\|f\|_{V,0}\leq C\|f\|_0$. This implies, for any $f\in C^0(\mathcal M; \Lambda_{1}^{d_s})$
$$  \sum_{i=1}^{d_u}\sum_{j=1}^{d_s}\|\chi_h\iota_{X_{u,h}^j}f\wedge Y_{u,h}^{i}\|_{V,0}\leq C'\sum_{i,j}\|X_{u,h}^j\|_{C^0}\|Y_{u,h}^i\|_{C^0}\|f\|_0\leq C''\|f\|_0$$
with a constant $C''$ independent of the choice of cover or partition of unity. This shows that $\|.\|_{V,1}$ takes finite values and a quick induction shows that it is also the case of $\|.\|_{V,k}$ for any $k$. The triangle inequality, homogeneity and non-negativity of $\|.\|_{V,k}$ follows from Lemma \ref{Crit} and an induction. Suppose now that $f \in C^0(\mathcal M; \Lambda_{k}^{d_s})$ and $f\neq 0$. Then $f$ is non zero on a small open set which is included in some $U_h$ of the open cover. Now, because $(X_{s,h}^j)$ and $(Y_{u,h}^i)$ are local basis, then there is $(i_0,j_0)$ such that $\chi_h\iota_{X_{s,h}^{j_0}}f\wedge Y_{u,h}^{i_0}$ is continuous and non zero on an open set. By induction and the fact that $m^u_V$ gives a positive measure to any non empty open set, we get that $\|f\|_{V,k}>0$ and $\|.\|_{V,k}$ thus defines a norm.

We now prove \eqref{eq:bbound}, for this, we consider $f\in C^{0}(\mathcal M; \Lambda_k^{d_s})$, a cover $\mathcal U\in \mathcal C$, a partition of unity $(\chi_h)$, $(X^j_{u,h})\in \mathscr V^u$ and $(Y^i_{u,h})\in \mathscr F^u$. For any $1\leq h \leq n$, we get
\begin{align*}
&\sum_{i=1}^{d_s}\sum_{j=1}^{d_u}\chi_h\iota_{X_{u,h}^j}e^{t\mathcal L_{\mathbf{X}}}f\wedge Y_{u,h}^{i}=\sum_{i=1}^{d_s}\sum_{j=1}^{d_u}\chi_h(y)e^{t\mathcal L_{\mathbf{X}}}\left( \iota_{(d\varphi_{-t})_{\varphi_{t}y}(X_{u,h}^j)}f\wedge e^{t\mathcal L_X}Y_{u,h}^{i}\right)(y)
\\&=\sum_{i=1}^{d_s}\sum_{j=1}^{d_u}e^{t\mathcal L_{\mathbf{X}}}\chi_h( \varphi_t y)\iota_{(d\varphi_{-t})_{\varphi_{t}y}(X_{u,h}^j)(\varphi_t y)}f(y)\wedge e^{t\mathcal L_X}Y_{u,h}^{i}(y).
\end{align*}
If we have $\mathcal U\in \mathcal C$, then using the invariance of the Anosov decomposition \eqref{eq:flowinv}, we see that $\varphi_{-t}\mathcal U\in \mathcal C$ with an associated partition of unity given by $(\chi_h\circ \varphi_t)_{1\leq h\leq n}$. Moreover, the same equation gives that $\big((d\varphi_{-t})_{\varphi_{t}y}(X_{s,h}^j)(\varphi_t y)\big)_{1\leq j\leq d_u}$ spans $E_s(y)$ for $y\in \varphi_{-t}U_h$ and similarly $\big(e^{t\mathcal L_X}Y_{u,h}^{i}(y)\big)_{1\leq i \leq d_s}$ spans $E_u^*(y)$. We compute
\begin{align*}
&\sum_{i=1}^{d_s}\sum_{j=1}^{d_u}e^{t\mathcal L_{\mathbf{X}}}\chi_h( \varphi_t y)\iota_{(d\varphi_{-t})_{\varphi_{t}y}(X_{u,h}^j)(\varphi_t y)}f(y)\wedge e^{t\mathcal L_X}Y_{u,h}^{i}(y)
\\=&\sum_{i,j}\|e^{t\mathcal L_X}Y_{u,h}^{i}\|\|(d\varphi_{-t})_{\varphi_{t}y}(X_{u,h}^j)\| e^{t\mathcal L_{\mathbf{X}}}\chi_h( \varphi_t y)\iota_{\tfrac{(d\varphi_{-t})_{\varphi_{t}y}(X_{u,h}^j)}{\|(d\varphi_{-t})_{\varphi_{t}y}(X_{u,h}^j)\|}}f(y)\wedge \frac{e^{t\mathcal L_X}Y_{u,h}^{i}}{\|e^{t\mathcal L_X}Y_{u,h}^{i}\|}.
\end{align*}
We then have $\|e^{t\mathcal L_X}Y_{u,h}^{i}\|,\|(d\varphi_{-t})_{\varphi_{t}y}(X_{u,h}^j)\|\leq Ce^{-t\eta}$ for some (uniform) $C,\eta>0$ by the Anosov property. Finally, we obtain 
$$ \sum_{i=1}^{d_s}\sum_{j=1}^{d_u}\|\chi_h\iota_{X_{u,h}^j}e^{t\mathcal L_{\mathbf{X}}}f\wedge Y_{u,h}^{i}\|_{V,k-1}\leq Ce^{-2\eta t} \sum_{i=1}^{d_s}\sum_{j=1}^{d_u}\big\|e^{t\mathcal L_{\mathbf{X}}}\left(\chi_h\iota_{\tilde X_{u,h}^j}f\wedge \tilde Y_{u,h}^{i}\right)\big \|_{V,k-1}$$
for $\tilde X_{u,h}^j:=(d\varphi_{-t})_{\varphi_{t}y}(X_{u,h}^j)/\|(d\varphi_{-t})_{\varphi_{t}y}(X_{u,h}^j)\|$ and $\tilde Y_{u,h}^{i}=e^{t\mathcal L_X}Y_{u,h}^{i}/\|e^{t\mathcal L_X}Y_{u,h}^{i}\|$
Passing to the supremum thus gives
$$\| e^{t\mathcal L_{\mathbf{X}}}f\|_{V,k-1}\leq Ce^{(P(V)-2(k-1)\eta) t}\|f\|_{V,k-1}\ \Rightarrow \  \| e^{t\mathcal L_{\mathbf{X}}}f\|_{V,k}\leq Ce^{(P(V)-2k\eta) t}\|f\|_{V,k}.$$
To conclude the proof of \eqref{eq:bbbound}, we only need to initialize the induction, and this is exactly the statement of Lemma \ref{Crit}.

Consider $f\in C^{\infty}(\mathcal M; \mathscr E_{d_s})$ and consider its decomposition
$$f= \sum_{k=0}^{d_s} \omega_k, \quad \omega_k \in C^{\alpha}(\mathcal M; \Lambda_k^{d_s}). $$
We define a norm on $C^{\infty}(\mathcal M; \mathscr E_{d_s})$ with the following bound when using the propagator:
$$\|f\|_{V}:=\sum_{k=0}^{d_s}\|\omega_k\|_{V,k}, \quad \|e^{t\mathcal L_X}f\|_V\leq Ce^{tP(V)}\|f\|_V.$$
This is the only thing we need to mimick the argument of Lemma \ref{negative} and Lemma \ref{pole} and prove that no resonance exists in $\{\mathrm{Re}(\lambda)>P(V)\}$ and that resonances on the critical axis have no Jordan block.
\end{proof}
We will now prove Theorem \ref{mainTheo} for the action on $d_s$-forms.
\begin{prop}[Critical axis for $d_s$-forms]
Under Assumption \ref{assumption}, the first resonance $P(V)$ is simple and the space of co-resonant states is spanned by $m^u_V$. 
\end{prop}
\begin{proof}
The argument of Lemma \ref{convergence} goes through and the projector $ \Pi_0(\lambda)$ on the resonant states is obtained as the limit of $\tilde R(\lambda)^n$ as an operator $\mathcal H^s \to \mathcal H^s$. The formula defining $\tilde R(\lambda)$ is the same. 
In particular, as in the proof of Proposition \ref{measure on the axis}, we obtain every co-resonant state by the following construction:
\begin{align*}
\forall \lambda\in P(V)+i\mathbb R, \ \forall v\in &C^{\infty}(\mathcal M; \Lambda^{d_s}(E_u^*\oplus E_s^*)), \   \Pi_0(\lambda)^*(v)=\lim_{n\to +\infty} (\tilde R(\lambda)^*)^n(v),  \end{align*}
where the convergence holds in $\mathcal H^s$.
Consider the pairing between $C^{\infty}(\mathcal M; \Lambda^{d_s}(E_u^*\oplus E_s^*))$ and $C^{\infty}(\mathcal M; \Lambda^{d_u}(E_u^*\oplus E_s^*))$ defined by $\langle v, u\rangle:=\int_{\mathcal M}u\wedge \alpha \wedge v,$ extended to  $\mathcal H^s\times \mathcal H^{-s}$.

\textbf{Co-resonant states define currents of order 0.} Consider two smooth forms $\omega\in C^{\infty}(\mathcal M; \Lambda^{d_s}(E_u^*\oplus E_s^*))$ and $\theta \in C^{\infty}(\mathcal M; \Lambda^{d_u}(E_u^*\oplus E_s^*))$. We first expand these forms:
\begin{equation}
\label{eq:splitt}
\omega=\sum_{k=0}^{d_s}\omega_k, \ \ \omega_k\in C^{\alpha}(\mathcal M; \Lambda_k^{d_s}), \quad \theta=\sum_{k=0}^{d_s}\theta_k, \ \ \theta_k\in C^{\alpha}(\mathcal M; \Lambda_k^{d_u}). 
\end{equation}
Note that the regularity of the component is only Hölder continuous here because of the regularity of the stable and unstable foliation. We see that $\omega_k\wedge \alpha \wedge \theta_l$ is a continuous $n$-form which is non zero if and only if $k+l=d_s$. 
The previous discussion justifies that the following convergence:
\begin{align*}
\langle \omega, \Pi_0(\lambda)^*\theta\rangle&=-\lim_{k\to+\infty}\int_{\mathbb R}(\chi'(t))^{(k)}e^{-t\lambda}\langle e^{t\mathbf P}\omega, \theta\rangle dt
\\&=- \lim_{k\to+\infty}\int_{\mathbb R}(\chi'(t))^{(k)}e^{-t\lambda}\sum_{k=0}^{d_s}\langle e^{t\mathbf P}\omega_k, \theta_{d_s-k}\rangle dt.
\end{align*}
Here, we have noticed that the integration over all the manifold is a $n$-current of order zero and thus extends to continuous $n$-forms. 

By Assumption \ref{assumption}, we can find a trivialization basis $(Y^j_{u})_{1\leq j \leq d_s}\in C^0(\mathcal M; E_u^*)$ such that $  E_u^*=\mathrm{Span}\{Y^j_{u}, \ 1\leq j \leq d_s\}$ and $(Y^j_{s})_{1\leq j \leq d_u}\in C^0(\mathcal M; E_s^*)$ such that $ E_s^*=\mathrm{Span}\{Y^j_{s}, \ 1\leq j \leq d_u\}$. Let $Y_s^1\wedge \ldots \wedge Y_s^{d_s}\wedge \alpha\wedge Y_s^1\wedge \ldots \wedge Y_s^{d_u}=: \tilde \omega\wedge \alpha  \wedge \tilde\theta$ be a non vanishing continuous $n$-form which we suppose to be positive. We notice to start that, using the Anosov property, there is an $\eta>0$ such that
\begin{align*}
\int_{\mathcal M}\big|&e^{t\mathbf{P}}(Y_u^1\wedge \ldots Y_u^{d_s-1}\wedge Y_s^1)\wedge \alpha \wedge Y_u^{d_s}\wedge Y_s^2\wedge \ldots \wedge Y_s^{d_s}\big|
\\&\leq  Ce^{-t\eta}\int_{\mathcal M}\big|e^{t\mathbf{P}}(Y_u^1\wedge \ldots Y_u^{d_s-1})\wedge \alpha \wedge Y_u^{d_s}\wedge Y_s^1\wedge  Y_s^2\wedge \ldots \wedge Y_s^{d_s}\big|
\\&\leq  Ce^{-2t\eta}\int_{\mathcal M}\big|e^{t\mathbf{P}}(Y_u^1\wedge \ldots Y_u^{d_s-1} \wedge Y_u^{d_s})\wedge \alpha \wedge Y_s^1\wedge  Y_s^2\wedge \ldots \wedge Y_s^{d_s}\big|
\\&\leq Ce^{-2t\eta}\int_{\mathcal M}e^{t\mathbf P}(\tilde \omega) \wedge \alpha \wedge \tilde \theta.
\end{align*}
Thus, a quick induction yields the following bound
$$\forall 0\leq k\leq d_s, \quad  |\langle e^{t\mathbf P}\omega_k, \theta_{d_s-k}\rangle|\leq C\|\omega\|_{C^0}e^{-2tk\eta} \langle e^{t\mathbf P}\tilde \omega,\tilde \theta\rangle.$$

Now, choose a smooth section $\omega$ and $\theta$ such that $\omega_0$ and $\theta_{d_s}$  are non vanishing on $\mathcal M$. Note that such an $\omega$ can be obtained by approximating in $C^0$ norm a non vanishing continuous section by smooth sections. Then, we get the convergence of 
$$ \ell:=\langle \omega, \Pi_0(\lambda)^*\theta\rangle =-\lim_{k\to+\infty}\int_{\mathbb R}(\chi'(t))^{(k)}e^{-t\lambda}\langle e^{t\mathbf P}\omega_0, \theta_0\rangle(1+O(e^{-\eta t})) dt.$$
Consider thus an arbitrary $\epsilon>0$ and an integer $k_0\geq 0$ such that for any $k\geq k_0$, one has
$$-\epsilon \leq \big|\int_{\mathbb R}(\chi'(t))^{(k)}e^{-t\lambda}\langle e^{t\mathbf P}\omega_0, \theta_0\rangle(1+O(e^{-\eta t})) dt+\ell\big| \leq \epsilon. $$
The support of the $(\chi'(t))^{(k)}$ goes to infinity with $k$ so we can suppose that for $k\geq k_0$, one has $ 1-\epsilon \leq 1+O(e^{-\eta t})\leq 1+\epsilon$ on the support of the integrand. In other words, we have obtained:
$$\forall k\geq k_0, \quad \frac{-\epsilon}{1+\epsilon} \leq  \big|\int_{\mathbb R}(\chi'(t))^{(k)}e^{-t\lambda}\langle e^{t\mathbf P}\omega_0, \theta_0\rangle dt+\ell\big|\leq \frac{\epsilon}{1-\epsilon}. $$
This also proves the boundedness of the above integral with $\omega_0$ and $\theta_0$ replaced by $\tilde \omega$ and $\tilde \theta$ and this shows that 
\begin{align*} \langle\omega, \Pi_0(\lambda)^*\theta\rangle& =\lim_{k\to+\infty}\int_{\mathbb R}(\chi'(t))^{(k)}e^{-t\lambda}\langle e^{t\mathbf P}\omega_0, \theta_0\rangle dt
\\ &\leq C\|\omega\|_0\limsup\int_{\mathbb R}|(\chi'(t))^{(k)}|e^{-tP(V)}\langle e^{t\mathbf P}\tilde \omega, \tilde \theta\rangle dt
\leq C'\|w\|_{C^0}.
\end{align*}
The above bound actually holds for any $\theta$ and $\omega$ and we get that every co-resonant state is an order zero current.

\textbf{Restrictions of co-resonant state on unstable manifolds are well defined.} Following the strategy of the proof of Proposition \ref{hard} we consider a co-resonant state $\theta$ which satisfies  $\mathrm{WF}(\theta)\subset E_s^*$. For any $x\in \mathcal M$, we can thus define the restriction of $\theta$ to $\mathcal W^u(x)$. We will denote by $\theta_x:=\theta_{|\mathcal W^u(x)}$. More explicitly, for a smooth function $f\in C^{\infty}(\mathcal M)$, we can define 
 $$(f_{|\mathcal W^u(x)},\theta_x):=(f[\mathcal W^u(x)], \theta)=(f[{\mathcal W^u(x)}], \theta)_{\mathcal H^s, \mathcal H^{-s}} $$
 where the bracket denotes the distributional pairing and $[\mathcal W^u(x)]$ denotes the $d_u$-current which consists in integrating over the unstable manifold $\mathcal W^u(x)$. Then $\theta_x$ is seen to be of order zero, in other words, $\{\theta_x\mid x\in \mathcal M\}$ defines a system of measures on the unstable leaves.
 
 \textbf{The system of measures  $\{\theta_x\mid x\in \mathcal M\}$ satisfies $\varphi_t$-conformality.} This is again a consequence of the fact that $\theta$ is a co-resonant state. More precisely, for any smooth function $f\in C^{\infty}(\mathcal M)$, we can write
 $$( f[\mathcal W^u(x)], \theta)_{\mathcal H^s \times \mathcal H^{-s}}=e^{-tP(V)}( e^{t\mathbf P}( f[\mathcal W^u(x)],\theta)_{\mathcal H^s \times \mathcal H^{-s}}.$$
Using $e^{t\mathbf P}( f[\mathcal W^u(x)])=e^{t\mathbf P} f[\mathcal W^u(\varphi_t x)]$, then yields
$$\theta_x(f)=e^{-tP(V)}\theta_{\varphi_t x}(e^{S_tV(\varphi_{-t}y)}f(\varphi_{-t}y)) $$
which is exactly the $\varphi_t$-conformality for the potential $V$.

\textbf{The system of measures $\{\theta_x \mid x\in \mathcal M\}$ satisfies the change of variable by holonomy.} This a consequence of the continuity of $x\mapsto \theta_x$ which is shown as in the proof of Proposition \ref{hard}. Indeed, another way of obtaining $\theta_x$ is to write $\theta$ as a linear combination of smooth $d_u$-form with distributional coefficients in $\mathcal H^{-s}$. Then the restriction $\theta_x$ is obtained by pulling back the smooth forms and the coefficients. The pullback on the coefficients are well defined by the wavefront set condition and we deduce the continuity property by the continuity of the coefficients.

Following the proof of Proposition \ref{hard} and using \cite[Corollary 3.12]{Cli}, we obtain the existence of a constant $c>0$ such that for any $x\in \mathcal M$, one has $\theta_x=c m^u_{x,V}$.

\textbf{The first resonance $P(V)$ is simple.} The first part of the proof shows that if suffices to prove that for any $\omega \in C^{\alpha}(\mathcal M; \Lambda_0^{d_s})$, one has $(\omega, \theta)=c(\omega, m^u_V)$. Note that such a $\omega$ writes $\omega =f\mathrm{vol}_{\mathcal W^u(x)}$ for some $f\in C^{0}(\mathcal M)$, we can then write
 $$\theta_x(f):=( f[\mathcal W^u(x)], \theta)_{\mathcal H^s \times \mathcal H^{-s}}=(  f\delta_{\mathcal W^u(x)},\mathrm{vol}_{\mathcal W^u(x)}\wedge \alpha \wedge \theta)_{\mathcal H^s \times \mathcal H^{-s}},$$
where $\delta_{\mathcal W^u(x)}$ is the integration of function on the unstable manifold of $x$ (see \eqref{eq:integr}). Remark that here, the last bracket is the usual distributional pairing. In particular, $\theta_x$ is the restriction to the unstable manifold $\mathcal W^u(x)$ of the measure $\mathrm{vol}_{\mathcal W^u}\wedge \alpha \wedge \theta$. Using the last part of the proof of Proposition \ref{hard} then yields that
$$\forall f\in C^{0}(\mathcal M ), \quad (\omega, \theta)=(\mathrm{vol}_{\mathcal W^u}\wedge \alpha \wedge \theta)(f)=c(\mathrm{vol}_{\mathcal W^u}\wedge \alpha \wedge m^u_V)=c(\omega, m^u_V).$$
This concludes the proof.
\end{proof}

We can now prove the second part  of Theorem  \ref{equilibriumtheo} for for $d_s$-forms.
\begin{prop}
The equilibrium state $\mu_{V}$ is given by the averaging formula \eqref{eq:average}.

Finally, if the flow is weak mixing with respect to the equilibrium state $\mu_{V}$, then the only resonance on the critical axis is $P(V)$.
\end{prop}

\begin{proof}
The averaging formula is derived form the product formula of Lemma \ref{lemmm} and the expression of the projector $ \Pi_0(P(V))$ using the exact same argument as in the proof of Proposition \ref{nothard}. Suppose now that there is another co-resonant state $\theta$ associated to a resonance $P(V)+i\lambda$ on the critical axis. We use the fact that $\theta$ is of order $0$ as well as \cite[Chapter 6 (2.6)]{Si} to get that there exists a Borel measure $\mu_{\theta}$ and a $\mu_{\theta}^\lambda$-mesurable function $T_{\theta}$ with values in $\Lambda^{d_s}(E_u^*\oplus E_s^*)$ normalized $\mu_{\theta}^\lambda$-a.e such that
$$\forall  \omega\in C^{\infty}(\mathcal M; \Lambda^{d_s}(E_u^*\oplus E_s^*)), \quad \langle\omega,\theta\rangle=\int_{\mathcal M}\langle \omega(x), T_{\theta}(x)\rangle d\mu^{\lambda}_{\theta}(x), $$
where the bracket here denotes the scalar product on $\Lambda^{d_s}(E_u^*\oplus E_s^*)$. Moreover, the measure $\mu_{\theta}^\lambda$ is defined by
\begin{equation}
\label{eq:mes}
\forall U\subset \mathcal M \text{ open set}, \ \ \mu^{\lambda}_{\theta}(U):=\sup_{\|w\|_{C^0}\leq 1, \mathrm{supp}(\omega )\subset U}\langle\omega, \theta\rangle.  
\end{equation}
The proof of the previous proposition showed that there is a $C>0$ such that 
$$\forall  \omega\in C^{\infty}(\mathcal M; \Lambda^{d_s}(E_u^*\oplus E_s^*)), \quad \langle\omega, \theta\rangle \leq C\langle \theta, m^u_V\rangle \ \Rightarrow \ \mu_{\theta}\ll  \mu_{m^u_V}=:\mu. $$

Now, using \eqref{eq:muv}, \eqref{eq:mes} as well as the bound 
$$\exists C>0, \forall q\in \mathcal M, \forall t\geq 0, \quad m^{u}_{V}(B_t(q,\delta))\leq C e^{S_t(V)(q)-tP(V)}, $$
(which can be found in the proof of \cite[Theorem 3.10]{Cli}), we get that $\mu$ satisfies an upper Gibbs bound. In other words, we have $\mu \ll \mu_V$. In particular, the Radon Nikodym density $h:=\tfrac{d\mu_\theta^\lambda}{d\mu}$ is an element of $L^{\infty}(\mathcal M, \mu_V)$. 

We use \cite[Theorem VII.14]{ReSi} to see that the flow is weakly mixing if and only if the only eigenvalue is $1$ and it is a simple eigenvalue. In other words, if 
\begin{equation}
\label{eq:thrid}
\begin{cases}
Xf=i\lambda f
\\ f\in L^2(\mathcal M,\mu_{V})
\end{cases}
\end{equation}
has no solution except for $\lambda=0$ and $f$ constant. But then, we have shown that $h$ is a non trivial solution of the above system which contradicts weak mixing.

Conversely, if the flow is not weakly mixing, then \eqref{eq:thrid} admits a $L^2$ solution which can be upgraded to a smooth solution $f$ by \cite[Theorem 4.3]{Wal}. Define $\gamma:=fm^u_V$, then 
$$\mathbf{P}^*\gamma=(P(V)+i\lambda)\gamma, \quad \mathrm{WF}(\gamma)\subset E_s^*, $$
that is, $P(V)+i\lambda$ is another resonance on $\mathcal C_{d_s}$.

\end{proof}

\section{Other critical axis and regularity of the pressure}
\label{sec6}
In this section, we will prove the statement about critical axes of Theorem \ref{mainTheo} as well as Corollary \ref{2}.
\begin{proof}
We start from the Guillemin trace formula
$$e^{\lambda t_0}\tr^{\flat}\big(\varphi_{-t_0}^*\big( (\mathbf{P})_{\mathscr E_0^{k}}-\lambda\big)^{-1}\big)=\sum_{\gamma\in \Gamma}\frac{T_{\gamma^{\sharp}}e^{-(\lambda+V_{\gamma}) T_{\gamma}}\tr(\Lambda^{k}\mathcal P_{\gamma})}{|\det(\mathrm{Id}-\mathcal P_{\gamma})|}.$$
We know by Lemma \ref{realnorm} that the critical axis of $(\mathbf{P})_{\mathscr E_0^{d_s}}$ is $\{\mathrm{Re}(\lambda)=P(V)\}$. We will prove that the other functions are analytic in a half plane $\{\mathrm{Re}(\lambda)>P(V)-\epsilon\}$ for some $\epsilon>0.$ We list the eigenvalues of the Poincare map 
\begin{equation}
e^{\lambda_1^-(\gamma)}\leq \ldots \leq e^{\lambda_{d_u}^-(\gamma)}\leq -e^{\eta T_{\gamma}}\leq  e^{\eta T_{\gamma}}\leq e^{\lambda_1^+(\gamma)}\leq \ldots \leq e^{\lambda_{d_s}^+(\gamma)} 
\end{equation}
for some uniform constant $\eta>0$ given by the Anosov property. Now we can compute
$$\tr(\Lambda^{k}\mathcal P_{\gamma}):=\sigma_{k}( e^{\lambda_1^-(\gamma)}, \ldots , e^{\lambda_{d_u}^-(\gamma)},e^{\lambda_1^+(\gamma)}, \ldots , e^{\lambda_{d_s}^+(\gamma)})$$
where $\sigma_k$ is the $k$-th symmetric polynomial. We see that the maximum value of $\tr(\Lambda^{k}\mathcal P_{\gamma})$ is attained at $k=d_s$ where one can choose all eigenvalues larger than $1$ without choosing any other eigenvalues. In particular, there is a constant $C>0$, independent of $\gamma$ and $k$ such that if $k\neq d_s$, then
$$|\tr(\Lambda^{k}\mathcal P_{\gamma})|\leq C\tr(\Lambda^{d_s}\mathcal P_{\gamma})e^{-\eta T_{\gamma}}. $$
But now, we know that for $k=d_s$, the trace is analytic and converges in $\{\mathrm{Re}(\lambda)>P(V)\}$. In particular, because all terms are positive, if $\epsilon<\eta$, then for any $\lambda\in (P(V)-\epsilon,P(V))$, one has
$$\sum_{\gamma\in \Gamma}\frac{T_{\gamma^{\sharp}}e^{-(\lambda+V_{\gamma}) T_{\gamma}}\tr(\Lambda^{k}\mathcal P_{\gamma})}{|\det(\mathrm{Id}-\mathcal P_{\gamma})|}\leq C\sum_{\gamma\in \Gamma}\frac{T_{\gamma^{\sharp}}e^{-(\lambda-\eta +V_{\gamma}) T_{\gamma}}\tr(\Lambda^{d_s}\mathcal P_{\gamma})}{|\det(\mathrm{Id}-\mathcal P_{\gamma})|}<+\infty $$
and this shows that the left hand side is analytic in the region $\{\mathrm{Re}(\lambda)>P(V)-\epsilon\}$ as claimed.
\end{proof}
Now we prove Corollary \ref{2}.
\begin{proof}
From \cite[Theorem 1]{GB}, we can find a $C^1$ neighborhood $\mathcal U$ of $(X_0, V_0)$ and an anisotropic Sobolev space $\mathcal H^{s_0}$ such that for any $(X,V)\in \mathcal U$, the operator $-X+V$ has discrete spectrum in $\{\mathrm{Re}(\lambda)\geq -1\}$. Now, the pressure $P_X(V+J^u)$ is obtained as a simple eigenvalue (from Theorem \ref{mainTheo}) of a smooth family of operators acting on $\mathcal H^{s_0}$. One deduces the regularity statement by a perturbation argument of the resolvent, just like in the proof of \cite[Corollary 2]{GB}. The regularity of $P_X(V)$ is obtained by the same argument but on $d_s$-forms.
\end{proof}
\section{Complex potential}
\label{seccomplex}
\begin{proof}
Consider a smooth complex valued potential $W=V+iU\in C^{\infty}(\mathcal M, \mathbb C)$ with $V,U\in C^{\infty}(\mathcal M, \mathbb C)$. We will only treat the case of $0$-froms since the proof is similar for $d_s$-forms. We start by proving \eqref{eq:inclusion}. Remark that since $iU$ is imaginary, the norm $\|.\|_V$ constructed in Lemma \ref{new} satisfies
\begin{equation}
\label{eq:lips2}
\forall f\in C^{\infty}(\mathcal M), \quad \|e^{t\mathbf{P}_W}f\|_V\leq e^{tP(V+J^u)}\|f\|_{V},
\end{equation}
where $\mathbf{P}_W:=-X+W=-X+V+iU$. This means that the proof of Lemma \ref{negative} extends to $W$ and this shows that \eqref{eq:inclusion}.

Actually, the estimate \eqref{eq:lips2} is the only thing needed for the proofs of Lemma \ref{pole}, Lemma \ref{convergence} and Proposition \ref{measure on the axis}. In other words, the Ruelle resonances on the critical axis $\mathrm{Res}_0(W)\cap \{ \mathrm{Re}(\lambda)= P(V+J^u)\}$ have no Jordan block and the coresponding resonant states are measures which are absolutely continuous with respect to $\eta_V$ where $\eta_V$ is the resonant state constructed in Theorem \ref{mainTheo} for the real potential $V$.

Let $P(V+J^u)+i\theta\in \mathrm{Res}_0(W)\cap \{ \mathrm{Re}(\lambda)= P(V+J^u)\} $ be a resonance on the axis and let $\gamma$ be a co-resonant state. Denote by $h=\tfrac{d\gamma}{d\eta}\in L^{\infty}(\mathcal M, \mu_{V+J^u})$ the Radon-Nikodym derivative of $\gamma$. Since one has
$$ (-X+V-P(V+J^u))\eta =0, \ (-X+V+iU-P(V+J^u)-i\theta)\gamma=0,$$
one deduces that $(-X(h)+i(U-\theta)h)\eta=0.$ Since $\eta$ has full support, integrating yields
\begin{equation}
    \label{eq:h}
    e^{-tX}h=e^{it(\theta-U_t)}h, \quad \mu_{V+J^u}- \text{a.e}.
\end{equation}
We can now use \cite[Theorem 4.3]{Wal} for the Lie group $G=S^1$ to deduce that there exists a solution of \eqref{eq:h} which is smooth. We can compute that
$$X(|h|^2)=2\mathrm{Re}(\bar{h}X(h))=2\mathrm{Re}(it(U-\theta)|h|^2)=0. $$
In particular, $h$ does not vanish since it is non-zero. Integrating \eqref{eq:h} on a periodic orbit $\gamma \in \Gamma$ gives \eqref{eq:arithm}. Conversely, if $U$ satisfies \eqref{eq:arithm}, then one can define a solution $h$ of \eqref{eq:h}
by fixing some $x_0\in \mathcal M$ and defining $h(\varphi_t x_0)=e^{it(U_t-\theta)}h(x_0)$. Using the Anosov closing lemma, one can show that this defines a continuous (hence $\mu_{V+J^u}$-measurable) function (see \cite[Theorem 10.1.12]{Lef} for a detailed argument). Then applying \cite[Theorem 4.3]{Wal} again gives that $h$ can be chosen to be smooth. Defining $\gamma=h\eta$, we obtain see that 
$$(-X+W-P(V+J^u)-i\theta)\gamma=0, \quad \mathrm{WF}(\gamma)\subset E_u^*, $$
i.e that $P(V+J^u)+i\theta\in \mathrm{Res}_0(W)$. This concludes the proof.

\end{proof}

\printbibliography[
heading=bibintoc,
title={References}]

\appendix
\section{Hölder continuity of a distributional product.}
\label{secappendix}
In this appendix, we fix a continuous function $\varphi$ and an order zero distribution $\theta $ such that $\mathrm{WF}(\theta)\subset E_s^*$. This means that the distributional product $(\delta_{\mathcal W^u(x)},\varphi \theta)$ is well defined.
We will prove the following result.

\begin{prop}
\label{appendix}
There exists $C(\varphi)>0,\alpha>0$ and $\epsilon>0$ such that if $d(x,x')<\epsilon$, then 
\begin{equation}
\label{eq:appendix} |(\delta_{\mathcal W^u(x)},\varphi \theta)-(\delta_{\mathcal W^u(x')},\varphi \theta)|\leq C(\varphi)d(x,x')^{\alpha}.
\end{equation}
Here, the exponent $\alpha$ only depends on the unstable foliation.
\end{prop}
\begin{proof}
The proof is mostly contained in the proof of \cite[Proposition 6]{Wei}. Indeed, to get \eqref{eq:appendix}, one needs to prove superpolynomial decay of the Fourier transform of $\delta_{\mathcal W^u(x)}-\delta_{\mathcal W^u(x')}$ for $\xi\notin \Omega$ where a small conic neighborhood of $E_0^*\oplus E_u^*$ with a constant that is Hölder continuous. In other words,
$$\forall \xi \notin \Omega, \ |\widehat{\delta_{\mathcal W^u(x)}}(\xi)-\widehat{\delta_{\mathcal W^u(x')}}(\xi)|\leq C_Nd(x,x')^{\alpha}\langle \xi \rangle ^{-N}$$
where the Fourier transforms should be, strictly speaking, defined locally using charts. The super polynomial decay is proved in \cite[Proposition 6]{Wei} using integration by parts. Inspection of the proof shows that all $C^k$ norms of the operator $L_l$ which is used to integrate by parts, depend Hölder continuously on the base point $x$. Weich only uses continuity in the proof to get uniform constant $C_N$ but using the Hölder continuity instead to bound the difference yields the refinement above.

On $\Omega$, we use the fact that the Fourier transform of $\theta$ decays super polynomially and that $|\widehat{\delta_{\mathcal W^u(x)}}(\xi)-\widehat{\delta_{\mathcal W^u(x')}}(\xi)|$ has order zero. Actually, using the Hölder continuity of the foliation, the above difference is again controled by $d(x,x')^{\alpha}$ and using the definition of the distributional product yields \eqref{eq:appendix}.
\end{proof}

\end{document}